\documentclass{amsart}
\usepackage{amsmath}
\usepackage{amssymb}
\usepackage[all]{xy}
\usepackage{graphicx}
\usepackage{mathrsfs}
\usepackage{pst-node}
\usepackage{tikz-cd} 
\usepackage{graphicx}
\usepackage[all]{xy}
\SelectTips{cm}{}

\usepackage[dvips]{epsfig}

\usepackage{pinlabel}

\numberwithin{equation}{section}

\usepackage{url}

\usepackage[latin1]{inputenc}
\usepackage{xspace,amssymb,amsfonts,euscript}
\usepackage{amsthm,amsmath}
\usepackage{palatino}
\usepackage{euscript}
\input xy \xyoption {all}

\usepackage{tikz}

\RequirePackage{color}
\definecolor{myred}{rgb}{0.75,0,0}
\definecolor{mygreen}{rgb}{0,0.5,0}
\definecolor{myblue}{rgb}{0,0,0.65}

\RequirePackage{ifpdf}
\ifpdf
  \IfFileExists{pdfsync.sty}{\RequirePackage{pdfsync}}{}
  \RequirePackage[pdftex,
   colorlinks = true,
   urlcolor = myblue, 
   citecolor = mygreen, 
   linkcolor = myred, 
   pagebackref,
   bookmarksopen=true]{hyperref}
\else
  \RequirePackage[hypertex]{hyperref}
\fi

\RequirePackage{ae, aecompl, aeguill} 

  \def\hg{{\mathfrak h}}

    \def\RM{{\mathbb{R}}}

    \def\ZM{{\mathbb{Z}}}

    \def\AC{{\mathcal{A}}}
    \def\BC{{\mathcal{B}}}

  \def\eb{{\mathbf e}}  
  \def\fb{{\mathbf f}}  \def\FC{{\mathcal{F}}}
    \def\GC{{\mathcal{G}}}
\def\HC{{\mathbf H}}    \def\HC{{\mathcal{H}}}
    \def\IC{{\mathcal{I}}}

    \def\LC{{\mathcal{L}}}
    \def\MC{{\mathcal{M}}}
    \def\NC{{\mathcal{N}}}

\def\a{\alpha}

\def\r{\rho}

\newcommand{\LL}{LL}
\newcommand{\oLL}{\overline{LL}}
\newcommand{\LLL}{\mathbb{LL}}

\newcommand{\co}{\colon}
\newcommand{\define}{\stackrel{\mbox{\scriptsize{def}}}{=}}

\newcommand{\nc}{\newcommand} \newcommand{\renc}{\renewcommand}

\newcommand{\rdots}{\mathinner{ \mkern1mu\raise1pt\hbox{.}
    \mkern2mu\raise4pt\hbox{.}
    \mkern2mu\raise7pt\vbox{\kern7pt\hbox{.}}\mkern1mu}}

\def\un{\underline}

\def\to{\rightarrow}

\def\longto{\longrightarrow}

\def\onto{\twoheadrightarrow}
\nc{\triright}{\stackrel{[1]}{\to}}
\nc{\longtriright}{\stackrel{[1]}{\longto}}

\nc{\Hb}{H^\bullet}

\nc{\Br}{\mathcal{\color{blue}b}}
\nc{\HotRR}{{}_R\mathcal{K}_R}
\nc{\HotR}{\mathcal{K}_R}
\nc{\excise}[1]{}
\nc{\defect}{\text{df}}
\nc{\h}[1]{\underline{H}_{#1}}
\mathchardef\mhyphen="2D

\nc{\Ga}{\mathbb{G}_a} 
\nc{\Gm}{\mathbb{G}_m} 

\nc{\Perv}{{\mathbf{P}}}

\nc{\IH}{{\mathrm{IH}}}

\nc{\ic}{\mathbf{IC}}

\nc{\gl}{{\mathfrak{gl}}}
\renc{\sl}{{\mathfrak{sl}}}
\renc{\sp}{{\mathfrak{sp}}}

\renc{\Im}{\textrm{Im}}

\nc{\HCM}{H^{BM}}

 \DeclareMathOperator{\Hom}{Hom}
\DeclareMathOperator{\gHom}{\Hom^\bullet}
 \DeclareMathOperator{\ch}{ch}
\DeclareMathOperator{\grk}{grk}
\DeclareMathOperator{\End}{End} 

\DeclareMathOperator{\id}{id}

\newtheorem{thm}{Theorem}[section]
\newtheorem{lem}[thm]{Lemma}

\newtheorem{prop}[thm]{Proposition}
\newtheorem{cor}[thm]{Corollary}

\newtheorem{Eg}[thm]{Example}

\theoremstyle{definition}
\newtheorem{defi}[thm]{Definition}
\newtheorem{notation}[thm]{Notation}

\theoremstyle{remark}
\newtheorem{remark}[thm]{Remark}

\DeclareMathOperator{\Ext}{Ext}

\def\Iff{\Longleftrightarrow}

\nc{\simto}{\stackrel{\sim}{\to}}

\nc{\SD}{\mathcal{H}_{\mathrm{BS}}}

\nc{\todo}[1]{{\color{red}{#1}}}
\nc{\gchange}[1]{{\color{blue}{GW: #1}}}
\nc{\nchange}[1]{{\color{magenta}{#1}}}

\title{The anti-spherical category}

 \author{Nicolas Libedinsky and Geordie Williamson}

\begin{document}


\begin{abstract}

We study a diagrammatic categorification (the ``anti-spherical
category'') of the
  anti-spherical module for any Coxeter group. We deduce
  that Deodhar's (sign) parabolic Kazhdan-Lusztig polynomials 
  have non-negative coefficients, and that a monotonicity conjecture
  of Brenti's holds. The main technical observation is a localization
  procedure for the anti-spherical category, from which we construct a ``light leaves'' basis of
  morphisms. Our techniques may be used to calculate many new elements
  of the $p$-canonical basis in the anti-spherical module.

\end{abstract}

\maketitle


\section{Introduction}

\subsection{} Kazhdan-Lusztig polynomials are remarkable polynomials associated to pairs
  of elements in a Coxeter group. They describe the base change matrix
  between the standard and the Kazhdan-Lusztig basis of the Hecke
  algebra. Since their discovery by Kazhdan and Lusztig in 1979, these
  polynomials have found applications throughout representation
  theory.

A fascinating aspect of the theory is that these polynomials
  are elementary to define and compute, however they also have deep
  properties that are far from obvious from their definition. For
  example, it was conjectured by Kazhdan and Lusztig in
  \cite{KLPolynomials} that these polynomials have non-negative
  coefficients. This conjecture was established soon after by Kazhdan
  and Lusztig \cite{KLSchubert} if the underlying Coxeter group is a
  finite or an affine Weyl group.  Kazhdan and Lusztig's
  conjecture was established in complete generality by Elias and the
  second author via Soergel bimodule techniques \cite{EW2}.

In 1987 Deodhar introduced parabolic
  Kazhdan-Lusztig polynomials \cite{De}. These polynomials are defined starting
  from the choice of a Coxeter group, a standard parabolic subgroup
  and a sign. They describe the base change matrix between the
  standard and Kazhdan-Lusztig basis of the spherical or
  anti-spherical (depending on the sign) module for the Hecke algebra. Kazhdan-Lusztig
  polynomials agree with parabolic Kazhdan-Lusztig polynomials for the
  choice of the trivial parabolic subgroup.
 Parabolic Kazhdan-Lusztig polynomials are also known to have deep
 representation theoretic and geometric significance. 
One of the two main theorems of this paper is the following:

  \begin{thm}\label{thm:pos}
    Parabolic Kazhdan-Lusztig polynomials associated to the
    sign representation have non-negative coefficients, for any Coxeter
    system and any choice of standard parabolic subgroup.
  \end{thm}

\newpage

Two remarks on this theorem:
\begin{enumerate}
\item  Kazhdan and Lusztig have  a theorem that identifies
  Kazhdan-Lusztig polynomials with the Poincar\'e polynomials of the
  stalks of intersection cohomology complexes on the flag variety. The parabolic analogue 
  of that theorem  was   given  in a beautiful paper by Kashiwara and Tanisaki
  \cite{KT} in 2002.  Thus the above
  theorem was already known for the case where both the  Coxeter group arises as the Weyl group of a
  symmetrisable Kac-Moody Lie algebra (this is the case if and only if  the order 
of the product of any two simple reflections belongs to the set   $\{ 2, 3, 4, 6, \infty \}$)  and  the standard parabolic subgroup is finite. 
\item The analogue of this theorem for the trivial representation is known if the standard parabolic subgroup is
finite. This is because parabolic Kazhdan-Lusztig polynomials associated to the trivial representation
are special cases of ordinary Kazhdan-Lusztig polynomials (see section \ref{Relations} for details).
We believe that the methods of this paper  can  be adapted to deduce the analogue of this
theorem for parabolic Kazhdan-Lusztig polynomials associated to the
trivial representation without the finiteness condition. 
\end{enumerate}

\subsection{} 
The proof that Kazhdan-Lusztig polynomials have non-negative
coefficients in \cite{EW2} relies on a detailed study of a
categorification of the Hecke algebra via certain bimodules
constructed by Soergel \cite{S90,SB}, which 
have come to be known as Soergel bimodules. The essential point
(``Soergel's conjecture'') is that the Kazhdan-Lusztig basis arises 
as the classes in the Grothendieck group of indecomposable Soergel
bimodules. Thus Soergel bimodules  provide a setting where
Kazhdan-Lusztig polynomials   have an interpretation as graded dimensions
of certain Hom spaces.

More
recently, Elias and the second author described the monoidal category of
Soergel bimodules by generators and relations \cite{EW}. The result is
a diagrammatically defined additive graded monoidal category which is
equivalent to the monoidal category of Soergel bimodules. In this
paper we work almost exclusively with this category, which we denote
$\HC$ and call the Hecke category.

It is natural to try to understand parabolic Kazhdan-Lusztig
polynomials by categorifying the modules in which they live. This is
precisely what we do in this paper for the anti-spherical module.

\subsection{}
Let $(W,S)$ be a Coxeter system, and let $H$ be its Hecke
algebra over $\ZM[v^{\pm 1}]$. Let $h_x$ denote its standard basis and
$b_x$ its canonical (or Kazhdan-Lusztig) basis. Fix a subset $I \subset S$ and let
${}^I W$ denote the set of minimal coset representatives for $W_I
\setminus W$. Let $N$
denote the anti-spherical (right) $H$-module
\[
N := \mathrm{sgn}_v \otimes_{H_I} H,
\]
where $\mathrm{sgn}_v$ denotes the quantized sign representation of $H_I$, the
standard parabolic subalgebra of $H$ determined by $I$. Let $n_x$ denote the standard
basis of $N$ and $d_x$ its Kazhdan-Lusztig basis.

Recall the Hecke category $\mathcal{H}$ from above. For any $w \in W$
there exists an indecomposable self-dual object $B_w \in \HC$ parametrized by
$w$.  Any indecomposable self-dual object in $\HC$ is isomorphic to
$B_w$ for some $w \in W$.  We have a canonical isomorphism of
$\ZM[v^{\pm 1}]$-algebras
\[
H \simto [\HC]
\]
defined on generators by $b_s \mapsto [B_s]$, for all $s \in
S$. Here we have employed the following
notation: given an additive graded (with shift functor $M \mapsto M(1)$) category $\MC$, let $[\MC]$ denote its split
Grothendieck group, which we view as a $\ZM[v^{\pm 1}]$-module via
$v[M] := [M(1)]$. Note that $[\HC]$ is an algebra because $\HC$ is a
monoidal category.

Now inside $\HC$ consider $\IC$ the additive category consisting of all
direct sums of shifts of $B_x$, for $x \notin {}^I W$. It turns out
that $\IC$ is a right tensor ideal of $\HC$ (i.e. if $X \in \IC$ and $B \in
\HC$ then $XB \in \IC$). In particular, if we consider the
quotient\footnote{By quotient we mean
the following: the objects of $\NC$ are the same as those of $\HC$;
a morphism is zero in $\NC$ if and only if it factors through an
object of $\IC$.} of additive categories
\[
\NC := \HC/\IC
\]
then this is a right module category over $\HC$. We call $\NC$ the
\emph{anti-spherical category} (associated to the subset $I \subset
S$). The following theorem justifies the name:

\begin{thm}\label{thm:N}
  There is a canonical isomorphism $N \simto [\NC]$ of $\ZM[v^{\pm
    1}]$-modules. This is 
  an isomorphism of right $H$-modules via the identification $H =
  [\mathcal{H}]$. Under this isomorphism, the indecomposable self-dual
  objects in $\NC$ correspond to the Kazhdan-Lusztig basis in $N$.
\end{thm}

We also prove a theorem giving a (``light leaves'') basis for the morphisms
between certain additive 
generators of $\NC$ (see Theorem \ref{DLT}). From this the
positivity of the corresponding parabolic Kazhdan-Lusztig polynomials
(Theorem \ref{thm:pos}) is an easy consequence. We also deduce (see Corollary \ref{Br}) from
these results a proof of a conjecture of Brenti \cite{M} on the
monotonicity of parabolic Kazhdan-Lusztig polynomials associated to
increasing subsets $I \subseteq J \subseteq S$.

\subsection{} We were also motivated in our study of the anti-spherical
category by representation theory. If $W$ is the Weyl group of a
complex semi-simple Lie algebra, the
anti-spherical category can be used to give a graded deformation of
parabolic category $\mathcal{O}$ (the subset $I \subset S$ is determined by the
parabolic subgroup appearing in the definition of parabolic category
$\mathcal{O}$).  This fact does not seem to be available explicitly in the
literature, however the papers \cite{Str} and \cite{KMS} contain results which are quite close.

The anti-spherical category is also important in modular
representation theory. 
Riche and the second author conjectured that the Hecke category acts
via translation functors on the principal block of representations of
an algebraic group \cite{RW}. This conjecture was proved in \cite{RW}
for $GL_n$, and has recently been established in general by
Bezrukavnikov-Riche \cite{BezRicheAction} and Ciappara
\cite{Ciappara}. Thus  the anti-spherical category sees all of the
(extremely 
subtle) representation theory of connected reductive algebraic
groups. (These developments were heavily motivated by earlier work of
Soergel \cite{SoeKL} and Arkhipov-Bezrukavnikov \cite{AB}.) In a
parallel development, Elias and Losev \cite{ELo} explained that one can use
singular Soergel bimodules to construct the categories of polynomial
representations of $GL_n$ together with the action of certain natural
endofunctors, in a purely combinatorial way. Their work provides further
evidence for the importance of the anti-spherical category in
modular representation theory. 

In \cite{RW} (the obvious analogue of) Theorem $\ref{thm:N}$ is proved
for the anti-spherical module of an affine Weyl group. (The
parabolic subgroup is taken to be the finite Weyl group.) The 
proofs there rely on geometry or representation theory in a crucial
way. One of the main motivations for the current work was to give
purely algebraic proofs of these basic statements, which work for any
Coxeter system. The proofs of the current paper involve
quite different technology than those of \cite{RW} and are simpler and
more general. 

\subsection{}
Another consequence of the conjectures of \cite{RW} is a character
formula for simple modules and indecomposable tilting modules for
reductive algebraic groups in characteristic $p$ in terms of the
$p$-canonical basis of the anti-spherical module. This conjecture was
first proved by Achar, Makisumi, Riche and the second author \cite{AMRW1,
  AMRW2}, and has recently be proved in greater generality\footnote{i.e. for
all weights, and all $p$} by Riche
and the second author \cite{RWST}. The paper \cite{ELo} of Elias and Losev
has related results for $GL_n$. 

The upshot is that the $p$-canonical basis in the anti-spherical
module contains the answers 
to several deep mysteries in the representation theory of algebraic
groups. However it is still not easy to compute. The third main
theorem of this paper (Theorem \ref{Q})   heuristically says that   the
localisation of the anti-spherical category is ``as simple as
possible''. This was unexpected for the authors because general cell
quotients of the Hecke category can have complicated endomorphism
rings (for a detailed example, see Elias' Temperley-Lieb quotient of
the Hecke algebra \cite{EliasTL}).

Theorem \ref{Q} is the technical heart
of the paper. It also forms the foundation for an effective algorithm to calculate the
$p$-canonical basis.
The basic idea is that via localisation one can reduce
calculations of the $p$-canonical basis in the anti-spherical module
(which can be performed via diagrammatics, as explained in \cite{JW}) 
to certain linear algebra problems over a polynomial ring (the ring denoted $R_I$ in \S \ref{sec:coinv}). 
In the special case  of an affine Weyl group,
with parabolic subgroup the finite Weyl group, the ring $R_I$ has one variable, but in general
it might have several variables.
This algorithm has been further developed and implemented by the
second author, Jensen and Gibson 
to provide a powerful new means to calculate characters of tilting
modules, and hence decomposition numbers for symmetric groups \cite{GJW}.
This produced new data that was key for  the production of the
``billiards conjecture'' by  Lusztig and the second author
\cite{LW}.


\subsection{} We conclude this introduction with a remark on positive
characteristic. In the body of this paper we work over a field of
characteristic zero. This is because our results rely crucially on
the so-called parabolic property of root systems (see \eqref{pp}),
which often fails for reflection representations of Coxeter groups in
positive characteristic. The parabolic property ensures that Theorem \ref{Q}  holds. It is an
interesting question as to what happens if one localises in settings
in which the parabolic property fails (as is the case for the
important example of the natural representation of affine Weyl groups
in characteristic $p$). At the time this paper was written,
  this question was mysterious to the authors. However, in the
  meantime there have been considerable advances in understanding this
  question, both from the algebraic side in the work of Hazi \cite{Hazi}, and on
  the geometric side via Smith-Treumann theory \cite{Treumann, LL,
    RWST}. The relations between these two theories, as well as why
  Smith-Treumann theory is relevant for describing
  localizations of the Hecke category in characteristic $p$ is explained in \cite{WilCDM}.

Finally, let us remark that one can still apply the techniques of
this paper to settings in positive characteristic by using the
$p$-adic integers in place of a field of characteristic $p$. This is
one of the basic ideas in the algorithm mentioned in the last
paragraph.

\subsection{Acknowledgements:} 
This paper owes an intellectual debt to
ideas of R. Bezrukavnikov and S. Riche. We would like to thank them
both. We would also like to
thank B. Leclerc for pointing out \cite{KT}.
Finally we would like to thank B.~Elias, E.~Gorsky, J.~Gibson, A.~Hazi, T.~Jensen and
P.~Sentinelli for interesting discussions and detailed comments on
various versions of this paper. The second author  was funded by ANID project Fondecyt regular
$1200061$.

\subsection{Note to the reader:}  A previous version of this article
(available on the arxiv) took a significantly more complicated route to our main
theorem, by exploiting properties of the infinite twist. This approach
contained gaps, pointed out by two referees. Whilst we believe that
our original approach still works, the referees' questions lead us to the
simplified proof presented here. We are very grateful to both referees. We remain interested in the possibilities of the infinite
twist, but omit discussion of it here. The authors learned much about the infinite twist from discussions with
M.~Hogencamp.

\section{Parabolic Kazhdan-Lusztig polynomials}

\subsection{The Hecke algebra} We follow the notation of \cite{SoeKL}. Let $(W,S)$ be a Coxeter system and $(m_{sr})_{s,r \in S}$ its Coxeter matrix. Let $l:W\rightarrow \mathbb{N}$ be the corresponding length function and $\leq$ the Bruhat order on $W$. Let $\mathcal{L}=\mathbb{Z}[v^{\pm1}]$ be the ring of Laurent polynomials with integer coefficients in one variable $v$.

The \emph{Hecke algebra} ${H}=H(W,S)$ of a Coxeter system $(W,S)$ is the associative algebra over $\mathcal{L}$ with generators $\{h_s\}_{s\in S}$, quadratic relations $(h_s+v)(h_s-v^{-1})=0$ for all $s\in S$, and  braid relations $h_sh_rh_s\cdots=h_rh_sh_r\cdots$ with $m_{sr}$ elements on each side for every couple $s, r \in S.$

Consider $x\in W.$ To a reduced expression $sr\cdots t$ of $x$ one can
associate the element $h_sh_r\cdots h_t\in H.$   It was proved by
H. Matsumoto that this element is independent of the choice of reduced
expression of $x$, and we call it $h_x.$ N. Iwahori proved that $$H=\bigoplus_{x\in W}\mathcal{L} h_x,$$
and $h_xh_y=h_{xy}$ if $l(x)+l(y)=l(xy)$ (which is clear by Matsumoto's Theorem).

Let us define the element $b_s=h_s+v.$ The right regular action of $H$ is given by the formula:
\begin{equation}\label{rightH}
  h_xb_s = \left\{
    \begin{array}{cl}
      h_{xs}+vh_x & \text{if } x<xs; \\
      h_{xs}+v^{-1}h_x & \text{if } x>xs .
          \end{array} \right.
\end{equation}

 \subsection{Parabolic subgroups}\label{ps}

 Consider $I\subset S$ an arbitrary subset and $W_I$ its corresponding  Coxeter group, which identifies naturally as a subgroup of $W$. We say that $W_I$ is the \emph{parabolic subgroup} corresponding to $I$. We say that a sequence $\underline{w}$ of elements in $S$ is an $I$-\emph{sequence} if it starts with some element $s\in I$.

 We denote by $^I\hspace{-0.03cm}W\subseteq W$ the set of minimal coset representatives in $W_I \hspace{-.1cm} \setminus \hspace{-.1cm}W.$  The following two descriptions of this set will be useful for us:
\begin{equation}\label{description1}^I\hspace{-0.03cm}W=\{w\in W \, \vert \, sw>w \text{ for all } s\in I \}  \text{;}
\end{equation}
\begin{equation}\label{description2}
^I\hspace{-0.03cm}W=\{w \in W \, \vert \, \text{ no reduced expression of } w \text{ is an } I\text{-sequence}\}.
\end{equation}

\begin{Eg} Let $W$ be the symmetric group $W=S_{8}$ with simple reflections $s_1, s_2,\ldots, s_7.$ For simplicity  we will just denote $s_k$ by $k$, so by  $343$ we mean the element $s_3s_4s_3\in W$. Let us define the set  $$\underrightarrow{54321}:= \{\emptyset, 5,54,543,5432,54321\}\subseteq W.$$ We define in the same way the set $\underrightarrow{k\ldots 321}$ for any natural number $k$. The order of this set is $k+1$. 

Say that $I=\{1,2,3\}.$ Then $W_I$ and $^I\hspace{-0.03cm}W$ are the following products of sets
 $$W_I=\underrightarrow{1}\ \underrightarrow{21}\ \underrightarrow{321}\ \   \text{ (it has order $2\cdot 3\cdot 4=24$) and}$$
$$ ^I\hspace{-0.03cm}W=\underrightarrow{4321}\ \underrightarrow{54321}\ \underrightarrow{654321}\ \underrightarrow{7654321} \ \text{(it has order } 5\cdot 6\cdot 7\cdot 8=1680).
$$
For example, $1\,21\,32\in W_I$ and $43\, 543\, 6\, 765432\in  ^I\hspace{-0.13cm}W$.

\end{Eg}
We see in this example (if one recalls the normal form of an element in the symmetric group) that multiplication defines an isomorphism of sets 
\begin{equation}\label{mult}
W\cong  W_I \times^I\hspace{-0.09cm}W
\end{equation} 
satisfying that, if $x\in W_I$ and $y\in \, ^I\hspace{-0.001cm}W,$ then $l(xy)=l(x)+l(y)$.
Deodhar \cite{De} proved that this is true for any Coxeter system and any parabolic subgroup.  
%

\subsection{Parabolic Property}\label{pp}
Let $\mathfrak{h}$ be the ``dual geometric representation'' of $W$ (see Section \ref{real}). Let $\Delta_I:=\{\alpha_r\}_{r\in I}\subset \hg^*$  and let $\Phi_I:=W_I\cdot \Delta_I$ be the root system spanned by $\Delta_I.$ Another  important property of minimal coset representatives is what we will call the \emph{Parabolic Property}:

  If $x\in ^I\hspace{-0.13cm}W$ and $s\in S,$ then   $$xs\notin {}^IW \Iff x(\alpha_s)\in \Phi_I.$$

\emph{Proof.}

\begin{itemize}
\item We first prove $\Rightarrow$
   If $x\in {}^IW$ and $xs\notin {}^IW$ then $xs=rx$ for some $r\in I.$
This comes from the more general (and beautiful) fact \cite[\S 3]{SoeKL} that if $x$ is
any element of $W$ and $s,r\in S$, the two inequalities $rx>x$ and
$rxs<xs$ imply that $rxs=x.$ Hence $x(\alpha_s)=rxs(\alpha_s)$. Thus we obtain the equality $r(x(\alpha_s))=-x(\alpha_s).$ This implies that $x(\alpha_s)=\alpha_{r}$ or $x(\alpha_s)=-\alpha_{r}$. 


\item Now we prove that ${x(\alpha_s)\in \Phi_I\Rightarrow xs\notin {}^IW}$.  As $x(\alpha_s)\in \Phi_I,$ we know by the Lemma in  \cite[\S $5.7$]{Hum} 
that $xsx^{-1}=t \in W_I$ with $t$ the reflection satisfying $x(\alpha_s) = \alpha_t$.  Rewriting this equation we have $xs=tx.$ Bijection \ref{mult} implies that $xs\notin {}^IW$. $\hfill \Box$
\end{itemize}

 \subsection{Spherical and anti-spherical modules}
 
We base the exposition and notations of the next sections in \cite{SoeKL}. Consider $I\subset S$ and the Hecke algebra $H_I:=H(W_I,I)$. 
By the relations defining the Hecke algebra, if we fix $u\in \{-v,v^{-1}\}$, we can define a surjection of $\mathcal{L}$-algebras $$\varphi_u:H_I\twoheadrightarrow \mathcal{L}$$  by sending $h_s\mapsto u$ for all $s\in I$. Thus $\mathcal{L}$ becomes an $H_I$-bimodule which we denote by $\mathcal{L}(u)$. We can induce from it to produce the following right $H$-modules:
$$N=N(W,S,I)=\mathcal{L}(-v)\otimes_{H_I}H, \text{ the } \textit{anti-spherical module};$$
$$M=M(W,S,I)=\mathcal{L}(v^{-1})\otimes_{H_I}H, \text{ the } \textit{spherical module}.$$

If $n_x:=1\otimes h_x\in N$ and $m_x:=1\otimes h_x\in M$, then we have that $$N=\bigoplus_{x\in ^I\hspace{-0.03cm}W} \mathcal{L}n_x  \ \ \text{ and } \ \  M=\bigoplus_{x\in ^I\hspace{-0.03cm}W} \mathcal{L}m_x.$$
 
 We will not prove this result but we will explain why it is reasonable. 
 Equality (\ref{description1}) tells us that if $x\notin ^I\hspace{-0.13cm}W$, then there is $r\in I$ such that $rx<x$, so   $n_x=-vn_{rx}.$ In this way we see that the set $\{n_x\}_{x\in ^I\hspace{-0.03cm}W}$ generates $N$ over $\mathcal{L}$ (a similar result holds for $M$).
 
 \subsection{Right action of the Hecke algebra.}\label{ra}
 
 The right action of $H$ on the anti-spherical and on the spherical modules (compare with the regular action (\ref{rightH})) is given by the formulas 
\begin{equation}\label{rightactionN}
  n_xb_s = \left\{
    \begin{array}{cl}
      n_{xs}+vn_x & \text{if } x<xs  \text{   and   }xs\in ^I\hspace{-0.12cm}W; \\
      n_{xs}+v^{-1}n_x & \text{if } x>xs \text{ and } xs\in ^I\hspace{-0.12cm}W; \\
      0 &\text{if } xs\notin ^I\hspace{-0.12cm}W.
    \end{array} \right.
\end{equation}

\begin{equation}
  m_xb_s = \left\{
    \begin{array}{cl}
      m_{xs}+vm_x & \text{if } x<xs  \text{   and   }xs\in ^I\hspace{-0.12cm}W; \\
      m_{xs}+v^{-1}m_x & \text{if } x>xs \text{ and } xs\in ^I\hspace{-0.12cm}W; \\
      (v+v^{-1})m_x & \text{if } xs\notin ^I\hspace{-0.12cm}W.
    \end{array} \right.
\end{equation}

Let us explain these formulas for the anti-spherical module. Similar arguments work in the spherical case. The  first two equations of (\ref{rightactionN}) are an easy consequence of (\ref{rightH}). The third equation of (\ref{rightactionN})
is a consequence of the following three facts: 
\begin{enumerate}
\item[(a)] $\varphi_{-v}(b_s)=0$ for $s\in I.$
\item[(b)]  If $x\in ^I\hspace{-0.12cm}W$ and $xs\notin ^I\hspace{-0.12cm}W$ then $xs=rx$ for some $r\in I.$
\item[(c)] If $x\in ^I\hspace{-0.12cm}W$ and $xs\notin ^I\hspace{-0.12cm}W$ then $xs>x$.
\end{enumerate}
Fact (a) is trivial. We have already seen facts (b) and (c) in \S \ref{pp}.

 \subsection{Kazhdan-Lusztig bases}
 
There is a unique ring homomorphism  $h \mapsto \overline{h}$ on $H$ such that $\overline{v}=v^{-1}$ and $\overline{h_x}=(h_{x^{-1}})^{-1}.$ Recall that $b_s=h_s+v.$ If $s\in I$ we have that $\varphi_{-v}(b_s)=0$ and  $\varphi_{v^{-1}}(b_s)=(v+v^{-1}).$ In any case $\varphi_u(\overline{b_s})=\overline{\varphi_u(b_s)}$ so, since the set $\{b_s\}_{s\in S}$ generates $H_I$ as an $\mathcal{L}$-algebra, we have 
\begin{equation}\label{hi}
\varphi_u(\overline{h_I})=\overline{\varphi_u(h_I)} \text{ for any element } h_I\in H_I.
\end{equation}
We also denote by $\overline{(-)}$ the involution of $\mathcal{L}$ given by $v\mapsto v^{-1}$.
Using equation (\ref{hi}), we can induce the morphism $\overline{(-)}$ to a
morphism of additive groups $\overline{(-)}:N\rightarrow N$ given by
$l\otimes h\mapsto \overline{l}\otimes \overline{h}$. In the same way
we can induce a morphism of additive groups
$\overline{(-)}:M\rightarrow M$. We will call an element
\emph{self-dual} if it is invariant under $\overline{(-)}$.

We can now state the central theorem of Kazhdan-Lusztig theory and its parabolic versions. 

\begin{thm}\label{KL}
\begin{enumerate}
\item(\cite{KLPolynomials}) For every element $x\in W$ there is a unique self-dual element $b_x\in {H}, $  such that  ${b}_x\in h_x+\sum_{y\in W}v\mathbb{Z}[v]h_y.$ 
\item(\cite{De}) For every element $x\in ^I\hspace{-0.22cm}W$ there is a unique self-dual element $c_x\in M, $  such that  $c_x\in m_x+\sum_{y\in ^I\hspace{-0.03cm}W}v\mathbb{Z}[v]m_y.$
\item(\cite{De}) For every element $x\in ^I\hspace{-0.22cm}W$ there is a unique self-dual element $d_x\in N, $  such that  $d_x\in n_x+\sum_{y\in ^I\hspace{-0.03cm}W}v\mathbb{Z}[v]n_y.$
\end{enumerate}
 \end{thm}

The sets $\{b_x\}_{x\in W}$, $\{c_x\}_{x\in ^I\hspace{-0.03cm}W}$ and $\{d_x\}_{x\in ^I\hspace{-0.03cm}W}$ are bases of the corresponding $H$-modules, and are called the \emph{Kazhdan-Lusztig bases}.  For each couple of elements $x,y\in W$ we define $h_{y,x}\in \LC$ by the formula $$b_x=\sum_yh_{y,x}h_y.$$ 
For each couple of elements  $x,y\in ^I\hspace{-0.16cm}W$  we define $m_{y,x}\in \LC$ and $n_{y,x}\in \LC$ by the formulae
$$c_x=\sum_{y\in ^I\hspace{-0.03cm}W}m_{y,x}m_y \ \ \text{ and }\ \ d_x=\sum_{y\in ^I\hspace{-0.03cm}W}n_{y,x}n_y.$$
(If we need to specify the  set $I$, we will write $m^I_{y,x}$ for $m_{y,x}$ and $n^I_{y,x}$ for $n_{y,x}$.)
The proof of Theorem \ref{KL} (as given by Soergel in \cite{SoeKL}) is short and easy. It constructs the Kazhdan-Lusztig basis   inductively on the length of $x$.

The Kazhdan-Lusztig polynomials (as defined in \cite{KLPolynomials}) are given by the formula $P_{y,x}=
(v^{l(y)-l(x)})h_{y,x}$ and they are polynomials in $q:=v^{-2}$. The same normalization gives Deodhar's parabolic polynomials. More precisely $(v^{l(y)-l(x)})m_{y,x}$ and $(v^{l(y)-l(x)})n_{y,x}$ are the polynomials $P^I_{y^{-1},x^{-1}}$ defined by Deodhar in \cite{De} in the cases $u=-1$ and $u=q$, respectively. 

\subsection{Some relations between these polynomials}\label{Relations}
\begin{enumerate}
\item In the case $I=\emptyset$ we have $H=M=N$, $b_x=c_x=d_x$ and $h_{y,x}=m_{y,x}=n_{y,x}$. Thus the theory of parabolic Kazhdan-Lusztig polynomials contains the theory of Kazhdan-Lusztig polynomials. 
\item If $I$ is finitary (i.e. $W_I$ is finite) then Deodhar \cite{De}
  proves that the $m$ polynomials are instances of Kazhdan-Lusztig
  polynomials. More precisely, he proves that if $w_0$ is the longest
  element of $W_I$ then $m_{y,x}=h_{w_0y,w_0x}.$ Moreover, $M$ is a
  sub-$H$-module of $H$ compatible with the duality.  

This result was expected. Parabolic Kazhdan-Lusztig polynomials
calculate (and this is their main reason to exist) the dimensions of
the intersection cohomology modules of Schubert varieties in $G/P$
where $G$ is a Kac-Moody group and $P$ is a standard
parabolic. Kazhdan-Lusztig polynomials calculate those dimensions in
the case of the flag variety $G/B$. When  $G$ is a semi-simple or
affine Kac-Moody group (and thus the parabolic subgroup of the Weyl
group of $G$ corresponding to $P$ is finite) one problem reduces to
the other, because one has a smooth fibration $G/B\rightarrow G/P$. 

 \item 
For arbitrary $I$ and $x,y \in  ^I\hspace{-0.12cm}W,$ Deodhar \cite{De} proved the formula $$n_{y,x}=\sum_{z\in W_I}(-v)^{l(z)}h_{zy,x}.$$
This  follows from the facts that, if $\pi$ is the obvious surjection $\pi: H\twoheadrightarrow N$ and $w=xy$ is the decomposition with $x\in W_I$ and $y\in ^I\hspace{-0.12cm}W,$ then we have $$\pi(h_{w})=(-v)^{l(x)}n_{y} \ \text{and}\ \pi(b_y)=d_y.$$ 
So, summarizing, $M$ is sometimes a good sub-object and $N$ is always a good quotient of   $H$ (seen as an $H$-module).
\item If $J\subseteq I$, then for all $y,x \in\,  ^I\hspace{-0.03cm}W$  $n_{y,x}^I\leq n_{y,x}^J$ (where $\le$
  denotes coefficientwise inequality). This is known
  as \emph{Brenti's monotonicity conjecture}. This conjecture was  stated by Francesco Brenti
in 2008 at the Conference ``Festive Combinatorics, Symposium in honor of Anders Bj\H{o}rner's 60th Birthday''.  We prove it in this paper (see Corollary \ref{Br})
  as a consequence of our main theorem.
\end{enumerate}


\section{The categories $\HC$, $\NC$ and $_Q\hspace{-0.03cm}\NC$} 
In this  section we define the Hecke category (denoted by $\HC$), the diagrammatic anti-spherical category (denoted by $\NC$) and a localization (denoted by $_Q\hspace{-0.03cm}\NC$). For the Hecke category we follow  the exposition given in \cite[\S 2.5-2.7]{HW}.
\subsection{Realizations} \label{real}
Recall that  a realization, as defined in \cite[\S 3.1]{EW}  consists of a commutative ring $\Bbbk$ and a
free and finitely generated $\Bbbk$-module $\hg$ together with subsets
\[
\{ \a_s \}_{s \in S} \subset \hg^* \quad \text{and} \quad \{ \a^\vee_s \}_{s \in S} \subset \hg
\]
of ``roots'' and ``coroots'' such that $\langle \alpha_s^\vee,
\alpha_s \rangle = 2$ for all $s \in S$, such that
  the formulas 
\[
s(v) := v - \langle v,\alpha_s  \rangle \alpha^\vee_s \hspace{.5cm} \mathrm{for }\  s \in S\ \mathrm{and }\ 
v \in \hg,
\]
define an action of $W$ on $\hg$ and such that a technical  condition on 2-colored quantum numbers (condition (3.3) in \cite[\S 3.1]{EW}) is satisfied.

Unless otherwise stated we will assume in this paper that $\hg$ is a realization where the parabolic property holds and such that the simple roots $\{
\alpha_s \} \subset \hg^*$ are linearly independent. Our basic example  of this is when 
$\Bbbk =
\RM$ and $\hg$ is the ``dual geometric representation'' of $W$, 
i.e.  we first choose a vector space $\mathfrak{h}$ with
$\mathfrak{h}^* = \bigoplus_{s \in S} \RM \alpha_s$, and then define
the elements $\{ \alpha^\vee_s \}_{s \in S} \subset \hg$ by the
equations
\begin{equation} \label{eq:justcos}
\langle \alpha_t^\vee, \alpha_s \rangle = -2\cos(\pi/m_{st})
\end{equation}
(by convention $m_{ss} =1$ and $\pi/\infty = 0$). Note that the subset
$\{ \alpha^\vee_s \}_{s \in S} \subset  \hg$ is linearly independent if and only if $W$ is finite (see the Theorem in \S 6.4 of \cite{Hum}). One can prove that the technical condition on 2-colored quantum numbers mentioned above is satisfied in this case using the analogue result for the geometric representation, because the quantum numbers of both realizations  agree.


Let  $R =S(\hg^*)$ be the ring of regular functions on $\hg$ or, equivalently, the symmetric algebra of $\hg^*$ over
$\Bbbk$. We see $R$ as a graded $\Bbbk$-algebra by declaring  $\deg \hg^*  =
2$. The action of  $W$ on $\hg^*$ extends to $R$ by functoriality.
For any $s \in S$, let  $\partial_s :
R \to R[-2]$ be the \emph{Demazure operator} defined by the formula
\[
\partial_s(f) = \frac{f - sf}{\alpha_s}.
\]
In \cite[\S
3.3]{EW} it is proved that this is well defined under our assumptions.

\subsection{Towards the morphisms in $\HC_{\mathrm{BS}}$} \label{subsec:diags}
An {\it $S$-graph} is a finite,  planar, decorated
graph with boundary properly embedded in the planar strip $\mathbb R \times [0,1]$.
Its edges are colored by $S$. The vertices in this graph are of 3 types:
\begin{enumerate}
\item univalent vertices (``dots''):
 $ \begin{array}{c}
\tikz[scale=0.5]{\draw[dashed] (3,0) circle (1cm);
\draw[color=blue] (3,-1) to (3,0);
\node[circle,fill,draw,inner sep=0mm,minimum size=1mm,color=blue] at (3,0) {};}
\end{array}$
\item trivalent vertices:
  $\begin{array}{c}
\tikz[scale=0.5]{\draw[dashed] (0,0) circle (1cm);
\draw[color=blue] (-30:1cm) -- (0,0) -- (90:1cm);
\draw[color=blue] (-150:1cm) -- (0,0);}
\end{array}$
\item $2m_{rb}$-valent vertices:
$\begin{array}{c}\tikz[scale=0.5,rotate=-11]{\draw[dashed] (0,0) circle (1cm);
\draw[color=blue] (0,0) -- (22.5:1cm);
\draw[color=blue] (0,0) -- (67.5:1cm);
\draw[color=blue] (0,0) -- (112.5:1cm);
\draw[color=blue] (0,0) -- (157.5:1cm);
\draw[color=blue] (0,0) -- (-22.5:1cm);
\draw[color=blue] (0,0) -- (-67.5:1cm);
\draw[color=blue] (0,0) -- (-112.5:1cm);
\draw[color=blue] (0,0) -- (-157.5:1cm);
\draw[color=red] (0,0) -- (0:1cm);
\draw[color=red] (0,0) -- (45:1cm);
\draw[color=red] (0,0) -- (90:1cm);
\draw[color=red] (0,0) -- (135:1cm);
\draw[color=red] (0,0) -- (180:1cm);
\draw[color=red] (0,0) -- (-45:1cm);
\draw[color=red] (0,0) -- (-90:1cm);
\draw[color=red] (0,0) -- (-135:1cm);
}\end{array}$
\\
We require that 
there are exactly
$2m_{{\color{red}r}{\color{blue}b}}<\infty$ edges originating from the vertex. They alternate in color between two different elements 
$r,b\in S$ around the vertex. The
pictured example has $m_{{\color{red}r}{\color{blue}b}} = 8$.
\end{enumerate}
Additionally any
$S$-graph may have its regions (the connected components of the complement
of the graph in $\mathbb R \times [0,1]$) decorated by boxes
containing homogenous elements of $R$.

The following is an example of an $S$-graph with $m_{{\color{blue}b},
  {\color{red}r}} = 5$, $m_{{\color{blue}b},
  {\color{green}g}} = 2$, $m_{{\color{green}g},
  {\color{red}r}} = 3$:
\[
\begin{tikzpicture}[scale=0.8]
  \coordinate (b) at (2.7,1);
  \coordinate (a) at (2,2);
\coordinate (c) at (3.7,2.5);
  \coordinate (d) at (0.5,3);
\coordinate (dd) at (-0,2);
  \coordinate (e) at (3,3.5);
  \coordinate (ed) at (3.5,3);
\coordinate (f) at (4.8,3.4);
\coordinate (z0) at (0.6,2.3);
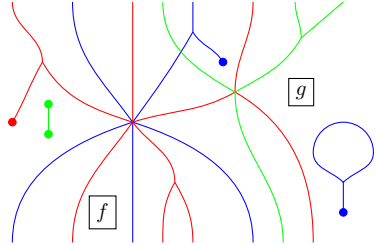
\coordinate (z1) at (0.6,1.8);

\node[rectangle,draw,scale=0.8] at (1.5,0.5) {$f$};
\node[rectangle,draw,scale=0.8] at (4.8,2.5) {$g$};

\draw[blue] (0,0) to[out=90,in=-162] (a);\draw[blue] (2,0) to (a);
\draw[blue] (4,0) to[out=90,in=-18] (a);  
\draw[blue] (a) to[out=126,in=-90] (1,4);\draw[blue] (a) to[out=54,in=-120] (e);
\draw[blue] (e) to (3,4);
\draw[blue] (ed) to[out=120, in=-60] (e);
 \node[circle,fill,draw,inner
      sep=0mm,minimum size=1mm,color=blue] at (ed) {};

\draw[red] (1,0) to [out=90,in=-126] (a) to[out=162,in=-60] (d) to[out=90,in=-90] (0,4);
\draw[red] (2.5,0) to[out=90,in=-120] (b) to[out=90,in=-54] (a) to[out=90,in=-90] (2,4);
\draw[red] (3,0) to[out=90,in=-60] (b);
\draw[red] (a) to[out=18,in=-150] (c) to[out=90,in=-90] (4,4);
\draw[red] (5,0) to[out=90,in=-30] (c);
\draw[red] (dd) to[out=60, in=-120] (d);
 \node[circle,fill,draw,inner
      sep=0mm,minimum size=1mm,color=red] at (dd) {};

\draw[green] (4.5,0) to[out=90,in=-90] (c) to[out=150,in=-90] (2.5,4);
\draw[green] (c) to[out=30,in=-120] (f) to[out=90,in=-90] (4.7,4);
\draw[green] (f) to (5.5,4);
\draw[green] (z0) to (z1);
 \node[circle,fill,draw,inner
      sep=0mm,minimum size=1mm,color=green] at (z0) {};
 \node[circle,fill,draw,inner
      sep=0mm,minimum size=1mm,color=green] at (z1) {};
\draw[blue] (5.5,1) to[out=135,in=-90] (5,1.5) to[out=90,in=180] (5.5,2) to[out=0,in=90] (6,1.5)
to[out=-90,in=45] (5.5,1) to[out=-90,in=90] (5.5,0.5);
 \node[circle,fill,draw,inner
      sep=0mm,minimum size=1mm,color=blue] at (5.5,0.5) {};
\end{tikzpicture}
\]
where $f$ and $g$ are homogeneous polynomials in $ R$.

The \emph{degree} of an $S$-graph is the sum over the degrees of
its vertices and boxes. Each box has degree equal to the degree of the
corresponding element of $R$. The vertices have degrees given by
the following rule: dots have degree $1$, trivalent vertices have degree
$-1$ and $2m$-valent vertices have degree 0. For example, the degree
of the $S$-graph above is $$+5-5+ \deg f + \deg g =
\deg f + \deg g.$$

The intersection of an  $S$-graph with  $\mathbb R \times\{0\}$ (resp. with $\mathbb R \times \{1\}$) is a sequence of colored points called
\emph{bottom boundary} (resp. \emph{top boundary}). In our example, the
bottom (resp. top) boundary of the $S$-graph  is $(b,r,b,r,r,b,g,r)$
(resp. $(r,b,r,g,b,r,g,g)$).

\subsection{Relations in $\HC_{\mathrm{BS}}$} Let us define the Hecke category. 
In this section we will give a summary of the central result of 
\cite{EW}.  

 We define $\SD$ as the monoidal category with objects 
 sequences $\un{w}$ in $S$.
If $\un{x}$ and $\un{y}$ are two such sequences, we define
$\Hom_{\SD}(\un{x}, \un{y})$ as the free $R$-module
generated by isotopy classes of $S$-graphs with bottom boundary $\un{x}$ and top boundary
$\un{y}$, modulo the local relations below. Hom spaces are graded by the degree of the graphs (all the relations below are homogeneous). The structure of this
monoidal category is given by horizontal concatenation of diagrams for the tensor product of morphisms and vertical concatenation
of diagrams for the composition of morphisms. 

In what follows, the rank of a relation is the number of colors involved in the relation. 
We use the color red for ${\color{red}r}$ and
blue for {$\color{blue}b$}. 
\vspace{.3cm}

\subsubsection{Rank 1 relations.}
\emph{Frobenius unit:}
\begin{gather}
  \begin{array}{c}
    \tikz[scale=0.7]{\draw[dashed] (0,0) circle (1cm);
\draw[color=red] (-90:1cm) -- (0,0) -- (90:1cm);
\draw[color=red] (-0:0.5cm) -- (0,0);
\node[circle,fill,draw,inner sep=0mm,minimum size=1mm,color=red] at (-0:0.5cm) {};
}
  \end{array}
=
  \begin{array}{c}
    \tikz[scale=0.7]{\draw[dashed] (0,0) circle (1cm);
\draw[color=red] (-90:1cm) -- (0,0) -- (90:1cm);}
  \end{array}. \label{2.7.1}
\end{gather}
 
\emph{Frobenius associativity:}
\begin{gather}\label{2.7.2}
  \begin{array}{c}
    \tikz[scale=0.7]{\draw[dashed] (0,0) circle (1cm);
\draw[color=red] (-45:1cm) -- (0.3,0) -- (45:1cm);
\draw[color=red] (135:1cm) -- (-0.3,0) -- (-135:1cm);
\draw[color=red] (-0.3,0) -- (0.3,0);
}
  \end{array}
=
  \begin{array}{c}
    \tikz[scale=0.7,rotate=90]{\draw[dashed] (0,0) circle (1cm);
\draw[color=red] (-45:1cm) -- (0.3,0) -- (45:1cm);
\draw[color=red] (135:1cm) -- (-0.3,0) -- (-135:1cm);
\draw[color=red] (-0.3,0) -- (0.3,0);
}
  \end{array}.
\end{gather}

\emph{Needle relation:}
\begin{gather}
  \begin{array}{c}
    \begin{tikzpicture}[scale=0.7]
      \draw[dashed] (0,0) circle (1cm);
      \draw[red] (0,0) circle (0.6cm);
      \draw[red] (0,0.6) --(0,1);
\draw[red] (0,-0.6) --(0,-1);
    \end{tikzpicture}
  \end{array}= 0.
\end{gather}

\emph{Barbell relation:}
\begin{gather}
  \begin{array}{c}
    \tikz[scale=0.7]{\draw[dashed] (0,0) circle (1cm);
\draw[color=red] (0,-0.6) -- (0,0.6);
\node[circle,fill,draw,inner sep=0mm,minimum size=1mm,color=red] at (0,0.6) {};
\node[circle,fill,draw,inner sep=0mm,minimum size=1mm,color=red] at (0,-0.6) {};
}
  \end{array}
=
  \begin{array}{c}
    \tikz[scale=0.7,rotate=90]{\draw[dashed] (0,0) circle (1cm);
\draw (-0.5,-0.5) rectangle (0.5,0.5);
\node at (0,0) {$\alpha_{\color{red} r}$};
}
  \end{array}.
\end{gather}

\emph{Nil Hecke relation:} 
\begin{gather}\label{eq:nilHecke}
  \begin{array}{c}
    \begin{tikzpicture}[scale=0.7]
      \draw[dashed] (0,0) circle (1cm);
      \draw[red] (0,1) --(0,-1);
\draw (0.2,-0.4) rectangle (0.7,0.4);
\node at (0.4,0) {$f$};
    \end{tikzpicture}
  \end{array}= 
  \begin{array}{c}
    \begin{tikzpicture}[scale=0.7]
      \draw[dashed] (0,0) circle (1cm);
      \draw[red] (0,1) --(0,-1);
\node at (-0.4,0) {${\color{red} r}f$};
\draw (-0.1,-0.4) rectangle (-0.7,0.4);
    \end{tikzpicture}
  \end{array}
+ 
  \begin{array}{c}
    \begin{tikzpicture}[scale=0.7]
      \draw[dashed] (0,0) circle (1cm);
      \draw[red] (0,1) --(0,0.6);\node[circle,fill,draw,inner
      sep=0mm,minimum size=1mm,color=red] at (0,0.6) {}; 
      \draw[red] (0,-1) --(0,-0.6);\node[circle,fill,draw,inner
      sep=0mm,minimum size=1mm,color=red] at (0,-0.6) {}; 
\node at (0,0) {$\partial_{\color{red} r}f$};
\draw (-0.5,-0.4) rectangle (0.5,0.4);
    \end{tikzpicture}
  \end{array}.
\end{gather}
(See \S \ref{real} for the definition of $\partial_r$.)

\vspace{.3cm}

\subsubsection{Rank 2 relations.}

\emph{Two-color associativity:} We give the first three cases i.e.
$m_{{\color{red}r}{\color{blue}b}} =2, 3, 4$. It is not hard to  guess this relation for arbitrary  $m_{{\color{red}r}{\color{blue}b}}$ (see
\cite[6.12]{EDC}  for  details).

$m_{rb} = 2$ (type $A_1 \times A_1$):
\begin{gather*}
  \begin{array}{c}
    \begin{tikzpicture}
      \draw[dashed] (0,0) circle (1cm);
\draw[red] (-45:1) -- (135:1);
\draw[blue] (-135:1) -- (45:0.3);
\draw[blue] (45:0.3) -- (20:1);
\draw[blue] (45:0.3) -- (80:1); 
    \end{tikzpicture}
  \end{array}
=
  \begin{array}{c}
    \begin{tikzpicture}
      \draw[dashed] (0,0) circle (1cm);
\draw[red] (-45:1) -- (135:1);
\draw[blue] (-135:1) -- (-135:0.4);
\draw[blue] (-135:0.4) -- (20:1);
\draw[blue] (-135:0.4) -- (80:1); 
    \end{tikzpicture}
  \end{array}
\end{gather*}

$m_{rb} = 3$ (type $A_2$):
\[ \begin{array}{c}
  \begin{tikzpicture}
      \draw[dashed] (0,0) circle (1cm);
\coordinate (a1) at (30:1);\coordinate (a2) at (60:1);\coordinate (a3) at (90:1);
\coordinate (a4) at (135:1);\coordinate (a5) at (-135:1);\coordinate
(a6) at (-90:1); \coordinate (a7) at (-45:1);

\coordinate (l1) at (45:0.4); \coordinate (l0) at (0:0);
\draw[red] (a5) -- (l0) -- (a3);
\draw[red] (l0) -- (a7);

\draw[blue] (a1) -- (l1) -- (a2);
\draw[blue] (l1) -- (l0) -- (a4);
\draw[blue] (l0) -- (a6);
  \end{tikzpicture}
\end{array}
= 
\begin{array}{c}
  \begin{tikzpicture}
      \draw[dashed] (0,0) circle (1cm);
\coordinate (a1) at (30:1);\coordinate (a2) at (60:1);\coordinate (a3) at (90:1);
\coordinate (a4) at (135:1);\coordinate (a5) at (-135:1);\coordinate
(a6) at (-90:1); \coordinate (a7) at (-45:1);

\coordinate (r1) at (120:0.4); \coordinate (r2) at (-60:0.4); \coordinate (r3) at (-135:0.7);

\draw[blue] (a2) to (r1) to (a4);
\draw[blue] (r1) to[out=-90,in=150] (r2) to (a6);
\draw[blue] (r2) to (a1);

\draw[red] (a3) to (r1) to[out=-150,in=90] (r3) to (a5);
\draw[red] (r1) to[out=-30,in=90] (r2) to (a7);
\draw[red] (r2) to[out=-150,in=0] (r3);

  \end{tikzpicture}
\end{array}
\]

$m_{rb} = 4$ (type $B_2$):
\[ \begin{array}{c}
  \begin{tikzpicture}
      \draw[dashed] (0,0) circle (1cm);
\coordinate (a0) at (30:1); \coordinate (a1) at (55:1);
\coordinate (a2) at (75:1);\coordinate (a3) at (105:1); \coordinate (a4) at (135:1);
\coordinate (a5) at (-45:1);\coordinate (a6) at (-75:1);
\coordinate (a7) at (-105:1); \coordinate (a8) at (-135:1);

\coordinate (l1) at (40:0.4); \coordinate (l0) at (0:0);

\draw[blue] (a0) -- (l1) -- (a1);
\draw[blue] (l1) -- (l0) -- (a3);
\draw[blue] (a6) -- (l0) -- (a8);

\draw[red] (a5) -- (l0) -- (a7);
\draw[red] (a2) -- (l0) -- (a4);

  \end{tikzpicture}
\end{array}
= 
 \begin{array}{c}
  \begin{tikzpicture}
      \draw[dashed] (0,0) circle (1cm);
\coordinate (a0) at (30:1); \coordinate (a1) at (55:1);
\coordinate (a2) at (75:1);\coordinate (a3) at (105:1); \coordinate (a4) at (135:1);
\coordinate (a8) at (-45:1);\coordinate (a7) at (-75:1);
\coordinate (a6) at (-105:1); \coordinate (a5) at (-135:1);

\coordinate (r1) at (120:0.4); \coordinate (r2) at (-60:0.4); \coordinate (r3) at (-135:0.7);

\draw[blue] (a1) -- (r1) -- (a3);
\draw[blue] (r1) to[out=-135,in=90] (r3) to (a5);
\draw[blue] (r3) to[out=0,in=-135] (r2) to (a7);
\draw[blue] (r1) to[out=-75,in=105] (r2);
\draw[blue] (r2) to (a0);

\draw[red] (a2) -- (r1) -- (a4);
\draw[red] (r1) to[out=-105,in=135] (r2)  to[out=75,in=-45] (r1);
\draw[red] (a6) to (r2) to (a8);

  \end{tikzpicture}
\end{array}
\]

\emph{Elias' Jones--Wenzl relation:} This relation expresses a dotted $2m_{rb}$-vertex
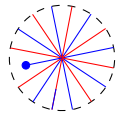
\[
\begin{array}{c}
    \tikz[scale=0.7,rotate=-11]{\draw[dashed] (0,0) circle (1cm);
\draw[color=blue] (0,0) -- (22.5:1cm);
\draw[color=blue] (0,0) -- (67.5:1cm);
\draw[color=blue] (0,0) -- (112.5:1cm);
\draw[color=blue] (0,0) -- (157.5:1cm);
\draw[color=blue] (0,0) -- (-22.5:1cm);
\draw[color=blue] (0,0) -- (-67.5:1cm);
\draw[color=blue] (0,0) -- (-112.5:1cm);
\draw[color=red] (0,0) -- (0:1cm);
\draw[color=red] (0,0) -- (45:1cm);
\draw[color=red] (0,0) -- (90:1cm);
\draw[color=red] (0,0) -- (135:1cm);
\draw[color=red] (0,0) -- (180:1cm);
\draw[color=red] (0,0) -- (-45:1cm);
\draw[color=red] (0,0) -- (-90:1cm);
\draw[color=red] (0,0) -- (-135:1cm);
\draw[color=blue] (0,0) -- (-157.5:0.7);
 \draw[blue] (0,-1) --(0,-0.6);\node[circle,fill,draw,inner
      sep=0mm,minimum size=1mm,color=blue] at (-157.5:0.7) {}; 
}
\end{array}
\]
as a linear combination over $R$ of diagrams consisting only of 
trivalent vertices and dots (no $2m_{{\color{red}r}{\color{blue}b}}$-valent vertices). We present again  the first three cases i.e.
$m_{{\color{red}r}{\color{blue}b}} =2, 3, 4$ (this time it is not easy to guess the general form, see
\cite[6.13]{EDC}  for all the details).

$m_{rb} = 2$ (type $A_1 \times A_1$):
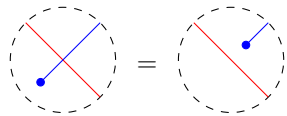
\begin{gather*}
  \begin{array}{c}
    \begin{tikzpicture}[scale=0.7]
      \draw[dashed] (0,0) circle (1cm);
\draw[red] (-45:1) -- (135:1);
\draw[blue] (45:1) -- (-135:0.6);
\node[circle,fill,draw,inner
      sep=0mm,minimum size=1mm,color=blue] at (-135:0.6) {}; 
    \end{tikzpicture}
  \end{array}
=
\begin{array}{c}
    \begin{tikzpicture}[scale=0.7]
      \draw[dashed] (0,0) circle (1cm);
\draw[red] (-45:1) -- (135:1);
\draw[blue] (45:1) -- (-135:-0.4);
\node[circle,fill,draw,inner
      sep=0mm,minimum size=1mm,color=blue] at (-135:-0.4) {}; 
    \end{tikzpicture}
  \end{array}
\end{gather*}

$m_{rb} = 3$ (type $A_2$):
\begin{gather*}
  \begin{array}{c}
    \begin{tikzpicture}[scale=0.7]
      \draw[dashed] (0,0) circle (1cm);
\foreach \r in {30,150,-90}
      \draw[red] (0,0) -- (\r:1cm);
\foreach \r in {90,-30}
      \draw[blue] (0,0) -- (\r:1cm);
\draw[blue] (0,0) -- (-150:0.7);
\node[circle,fill,draw,inner
      sep=0mm,minimum size=1mm,color=blue] at (-150:0.7) {}; 
    \end{tikzpicture}
  \end{array}
=
  \begin{array}{c}
    \begin{tikzpicture}[scale=0.7]
      \draw[dashed] (0,0) circle (1cm);
\foreach \r in {30,150,-90}
      \draw[red] (0,0) -- (\r:1cm);
\foreach \r in {90,-30}
{ \draw[blue] (\r:0.5) -- (\r:1cm);
\node[circle,fill,draw,inner sep=0mm,minimum size=1mm,color=blue] at (\r:0.5) {};}
    \end{tikzpicture}
  \end{array} + 
  \begin{array}{c}
    \begin{tikzpicture}[scale=0.7]
      \draw[dashed] (0,0) circle (1cm);
\draw[red] (-90:1) to[out=90,in=-30] (150:1);
\draw[blue] (90:1) to[out=-90,in=150] (-30:1);
\draw[red] (30:0.5) -- (30:1cm);
\node[circle,fill,draw,inner sep=0mm,minimum size=1mm,color=red] at (30:0.5) {};
    \end{tikzpicture}
  \end{array} 
\end{gather*}

$m_{rb} = 4$ (type $B_2$):
\begin{gather*}
  \begin{array}{c}
    \begin{tikzpicture}[scale=0.6,rotate=-22.5]
      \draw[dashed] (0,0) circle (1cm);
\foreach \r in {0,90,180,-90}
      \draw[red] (0,0) -- (\r:1cm);
\foreach \r in {45,135,-45}
      \draw[blue] (0,0) -- (\r:1cm);
\draw[blue] (0,0) -- (-135:0.6);
\node[circle,fill,draw,inner
      sep=0mm,minimum size=1mm,color=blue] at (-135:0.6) {}; 
    \end{tikzpicture}
  \end{array} = 
  \begin{array}{c}
    \begin{tikzpicture}[scale=0.6]
      \def\r{0.6}
      \draw[dashed] (0,0) circle (1cm);
 \foreach\a in {-90,45,135}
\draw[blue] (\a:1) to (\a:\r);
 \foreach\a in {90,-45}
\draw[red] (\a:1) to (\a:\r);
\draw[red] (-120:1) to[out=80,in=-45] (157.75:1);
\draw[blue] (-90:\r) to[out=90,in=-45] (135:\r);
\draw[red] (90:\r) to[out=-90,in=135] (-45:\r);
\node[circle,fill,draw,inner
      sep=0mm,minimum size=1mm,color=blue] at (45:\r) {}; 
    \end{tikzpicture}
  \end{array}
+ 
  \begin{array}{c}
    \begin{tikzpicture}[scale=0.6]
      \def\r{0.6}
      \draw[dashed] (0,0) circle (1cm);
 \foreach\a in {-90,45,135}
\draw[blue] (\a:1) to (\a:\r);
 \foreach\a in {90,-45}
\draw[red] (\a:1) to (\a:\r);
\draw[red] (-120:1) to (157.75:1);
\draw[red] (-135:\r) to (-135:0.85);
\draw[blue] (-90:\r) to[out=90,in=-135] (45:\r);
\draw[red] (-135:\r) to[out=45,in=-90] (90:\r);
\node[circle,fill,draw,inner
      sep=0mm,minimum size=1mm,color=blue] at (135:\r) {}; 
\node[circle,fill,draw,inner
      sep=0mm,minimum size=1mm,color=red] at (-45:\r) {}; 
    \end{tikzpicture}
  \end{array}
+  
\begin{array}{c}
    \begin{tikzpicture}[scale=0.6]
      \def\r{0.6}
      \draw[dashed] (0,0) circle (1cm);
 \foreach\a in {-90,45,135}
\draw[blue] (\a:1) to (\a:\r);
 \foreach\a in {90,-45}
\draw[red] (\a:1) to (\a:\r);
\draw[red] (-120:1) to (157.75:1);
\draw[red] (-135:\r) to (-135:0.85);
\draw[blue] (135:\r) to[out=-45,in=-135] (45:\r);
\draw[red] (-135:\r) to[out=45,in=135] (-45:\r);
\node[circle,fill,draw,inner
      sep=0mm,minimum size=1mm,color=red] at (90:\r) {}; 
\node[circle,fill,draw,inner
      sep=0mm,minimum size=1mm,color=blue] at (-90:\r) {}; 
    \end{tikzpicture}
  \end{array}
+  
\sqrt{2} \begin{array}{c}
    \begin{tikzpicture}[scale=0.6]
      \def\r{0.6}
      \draw[dashed] (0,0) circle (1cm);
 \foreach\a in {-90,45,135}
\draw[blue] (\a:1) to (\a:\r);
 \foreach\a in {90,-45}
\draw[red] (\a:1) to (\a:\r);
\draw[red] (-120:1) to (157.75:1);
\draw[red] (-135:\r) to (-135:0.85);
\foreach\a in {-135,-45,90}
\draw[red] (\a:\r) to (0,0);
\node[circle,fill,draw,inner
      sep=0mm,minimum size=1mm,color=blue] at (45:\r) {}; 
\node[circle,fill,draw,inner
      sep=0mm,minimum size=1mm,color=blue] at (-90:\r) {}; 
\node[circle,fill,draw,inner
      sep=0mm,minimum size=1mm,color=blue] at (135:\r) {}; 
    \end{tikzpicture}
  \end{array}
+  
\sqrt{2}\begin{array}{c}
    \begin{tikzpicture}[scale=0.6]
      \def\r{0.6}
      \draw[dashed] (0,0) circle (1cm);
 \foreach\a in {-90,45,135}
\draw[blue] (\a:1) to (\a:\r);
 \foreach\a in {90,-45}
\draw[red] (\a:1) to (\a:\r);
\draw[red] (-120:1) to[out=80,in=-45]  (157.75:1);
\foreach\a in {135,45,-90}
\draw[blue] (\a:\r) to (0,0);
\foreach\a in {-45,90}
\node[circle,fill,draw,inner
      sep=0mm,minimum size=1mm,color=red] at (\a:\r) {}; 
    \end{tikzpicture}
  \end{array}
\end{gather*}

\vspace{.3cm}

\subsubsection{Rank 3 relations.} We will not repeat the definition of the Zamolodchikov relations here, and instead refer the
reader to \cite[\S 1.4.3]{EW}.
This concludes the definition of $\HC_{\mathrm{BS}}$.

\subsection{The categories $\HC$ and $\HC_I$}
{

If $M=\bigoplus_iM^i$ is a $\mathbb{Z}$-graded object, we denote by $M(1)$
its grading shift, i.e. $M(1)^i=M^{i+1}.$  If $p=\sum_ja_jv^j\in  \ZM_{\geq 0}[v^{\pm 1}]$, we denote $$p\cdot M=\oplus_j M(j)^{\oplus a_j}.$$

Given an additive category
$\AC$ we denote by $[\AC]$ its split Grothendieck group.
If in addition $\AC$ has homomorphism spaces enriched in graded vector
spaces we denote by $\AC^\oplus$ its additive graded envelope. That
is, objects are formal finite direct sums $\bigoplus a_i(m_i)$ for certain objects
$a_i \in \AC$ and ``grading shifts'' $m_i \in \ZM$. Homomorphism
spaces in $\AC^\oplus$ are given by
\[
\Hom_{\AC^\oplus}( \bigoplus a_i(m_i), \bigoplus a_j'(m_j')) := \bigoplus
\Hom(a_i, a_j')(m_j' - m_i).
\]
We denote by $\AC^{\oplus, 0}$ the category with the same objects as
$\AC^{\oplus}$ but with homomorphism spaces given by the degree zero
morphisms in $\AC^{\oplus}$:
\[
\Hom_{\AC^{\oplus,0}}( b, b) :=\Hom_{\AC^\oplus}( b, b')^0.
\]
Both $\AC^\oplus$ and $\AC^{\oplus,0}$ are equipped with a grading shift
functor $b \mapsto b(1)$ given on objects by $\bigoplus a_i(m_i) \mapsto
\bigoplus a_i(m_i+1)$. Of course $\AC^{\oplus}$ is recoverable from
$\AC^{\oplus, 0}$ and the grading shift functor $(1)$. Finally we define
$\AC^e$ to be the Karoubian envelope of 
$\AC^{\oplus,0}$. In this setting, given objects $b, b' \in \AC^e$ we abbreviate:
\begin{gather*}
  \Hom(b,b') := \Hom_{\AC^e}(b,b'), \\
\gHom(b,b') := \bigoplus_{m \in \ZM} \Hom( b, b'(m)).
\end{gather*}

By definition the Hecke category $\HC$ is $\HC_{\mathrm{BS}}^e.$ 
If $I\subset S$, we define $\HC_I := \HC_{\mathrm{BS},I}^e$ where
$\HC_{\mathrm{BS},I}$ is the diagrammatic category
obtained by replacing $S$ by $I$ in the definition. (That is, all
diagrams in $\HC_{\mathrm{BS},I}$ are only allowed to be colored by
elements of $I$ and decorated by the symmetric algebra of $\oplus_{s\in I}\mathbb{R}\alpha_s$ over $\mathbb{R}$.) }



\subsection{Basic facts about $\HC$}\label{basic}

 Let us recall some terminology and notations from \cite{EW}. A
 \emph{subsequence} of an expression $\un{x} = s_1s_2\dots s_m$ is a sequence $\pi_1 \pi_2 \dots \pi_m$ such that $\pi_i \in \{ e, s_i \}$ for all $1 \le i \le m$. Instead of working with subsequences,
we work with the equivalent datum of a sequence $\eb = \eb_1 \eb_2 \dots \eb_m$ of 1's and 0's giving the indicator function of a subsequence, which we refer to as a \emph{01-sequence}. For an expression  $\un{x} = s_1s_2\dots s_m,$ we use the notation $\eb \subseteq \underline{x}$ if the $01$-sequence  $\eb$ has exactly $m$ terms. 

The \emph{Bruhat stroll} is the sequence $x_0=e, x_1,
\dots, x_m$ defined by \[ x_i := s_1^{\eb_1} s_2^{\eb_2} \dots s_i^{\eb_i} \] for $0 \le i \le m$. We call $x_i$ the $i^{\mathrm{th}}$-\emph{point} and $x_m$ the \emph{end-point} of the Bruhat stroll. We denote $x_m$ by $\un{x}^{\eb}$. Alternatively, we will say that a subsequence $\eb$ of $\un{x}$ \emph{expresses} the end-point $\un{x}^{\eb}$.

 Let $\eb$ and $\fb$ be two 01-sequences  of $\underline{x}=s_1s_2\cdots s_m$ and let their corresponding Bruhat strolls be $x_0, x_1, \dots, x_m$ and $y_0, y_1, \dots, y_m$. We say that $\eb\geq \fb $ in the \emph{path dominance order}  if $x_i \ge y_i$ for all $0 \le i \le m$. We define the \emph{double path dominance order} (a partial order) on pairs $(\eb,\fb)$, where $(\eb_1,\fb_1)\leq (\eb_2,\fb_2)$ if $\eb_1\leq \eb_2$ and $\fb_1\leq \fb_2$.


\emph{ Light leaves}  and \emph{Double leaves} for Soergel bimodules
were introduced  in \cite{LLL} and \cite{LLC}. They give bases, as $R$-modules of the Hom spaces between
Bott-Samelson bimodules. We recommend
reading the  paper \cite{LLC} in order to get used  to these combinatorial
objects and to read \S 6.1--6.3  of \cite{EW}, where these bases are explained diagrammatically.

In \cite[Definition 6.24]{EW} the authors define a character map $\ch:
[\HC] \to H$ and in  \cite[Corollary 6.27]{EW} they prove that it  is an isomorphism. This is the reason why we call  $\HC$ the
Hecke category.

 Following Soergel's
classification of indecomposable Soergel bimodules, in \cite[Theorem
6.26]{EW} the authors prove that the indecomposable objects in $\HC$
are indexed by $W$ modulo shift, and they call $B_w$ the
indecomposable object corresponding to $w\in W$.  It happens that the object $B_s$ is the sequence with one element $(s)\in \HC.$ Because of this, if $\underline{w}=(s,r,\cdots, t)$ we will sometimes denote by $B_{\underline{w}}:=B_sB_r\cdots B_t$  the element $\underline{w}\in \HC.$

 Let us suppose until the end of Section \ref{basic}  that $\hg$ is our favorite example,  the dual geometric representation over $\mathbb{R}$. Let us refer to the reflection faithful representation of $W$ over $\mathbb{R}$ that Soergel constructs \cite[\S 2]{Soe} as the \emph{Kac-Moody representation} $V_{\mathrm{KM}}.$ The  representation  $V_{\mathrm{KM}}$ is self-dual. By definition of $V_{\mathrm{KM}}$, we have $\hg^*\subset V_{\mathrm{KM}}.$ Using the mentioned self-duality, we obtain an injection of $W$-representations $i: \hg^*\hookrightarrow V_{\mathrm{KM}}^*.$ This extends to an injection of symmetric algebras $R =S(\hg^*)\hookrightarrow S(V_{\mathrm{KM}}^*).$ This means that one can see the diagrammatic Hecke category $\HC$ associated to the dual geometric representation as a  subcategory of the diagrammatic Hecke category associated with the Kac-Moody representation, that we denote $\HC(V_{\mathrm{KM}})$. The latter category is equivalent to the category of Soergel bimodules  $\BC(V_{\mathrm{KM}})$ as proved in \cite{EW}. This, and the main result of \cite{EW2} imply that $\ch([B_w]) = b_w$. 
 Thus the indecomposable objects in $\HC$ categorify the
 Kazhdan-Lusztig basis.

\subsection{The anti-spherical category $\NC$}\label{ASCN} Fix a subset $I \subset S$.
We define  the \textit {Bott-Samelson anti-spherical category}
$\NC_{\mathrm{BS}}$ to be the category $\HC_{\mathrm{BS}}$ quotiented
by the ideal of all objects indexed by $I$-sequences. The \textit
{anti-spherical category} $\NC$ is the graded additive Karoubian
completion of $\NC_{\mathrm{BS}}$, i.e.   $\NC := \NC_{\mathrm{BS}}^e.$
For $x\in ^I\hspace{-0.13cm}W$ we call  $D_x$ the image of $B_x$ in
the anti-spherical category $\NC$ and for $\underline{y}$ and
expression, we call $N_{\underline{y}}$ the image in $\NC$ of the
object $B_{\underline{y}}$, or in other words,
$N_{\underline{y}}=N_{\id}\cdot B_{\underline{y}} $, where $N_{\id}$
is the image of the empty sequence.

We define the category $\NC'$ to be the category $\HC$ quotiented by the ideal of all objects $B_x\in \HC$, with $x\notin ^I\hspace{-0.13cm}W.$

\begin{prop}\label{prop:N}
 There is an equivalence of categories $\NC\cong \NC'.$
\end{prop}

\begin{proof}
Consider the  monoidal functor $\FC_1: \HC^{\oplus,0}_{\mathrm{BS}} \to \NC'$  defined as
the composition of the inclusion functor $\HC^{\oplus,0}_{\mathrm{BS}}
\hookrightarrow \HC$  with the canonical projection $\HC \to \NC'$.  

Let $s\in S$ and $x\in W$ be such that $sx>x$. If   $b_sb_x=b_{sx}+\sum_{y<sx} m_yb_y,$   we have that
$m_y\in \mathbb{Z}_{\geq 0}$ and that $m_y\neq 0\Rightarrow sy<y.$ Let $\underline{w}$ be an $I$-sequence, say $\underline{w}=(s,s_1\ldots, s_n)$ with $s\in I$. Consider the decomposition of the sequence $(s_1\ldots, s_n)$ into indecomposable summands $\oplus_z\, p_z\cdot B_z$, with $p_z\in \ZM_{\geq 0}[v^{\pm 1}]$. This gives a decomposition  $\underline{w}=\oplus_z\,   p_z\cdot B_sB_z$. This we can rewrite as $\underline{w}=\oplus_u\, p'_u\cdot B_u$, with $p'_u\in \ZM_{\geq 0}[v^{\pm 1}]$.
 Every $B_u$ appearing in a non-zero term of this sum is such that $su<u,$ thus $u\notin ^I\hspace{-0.13cm}W,$ and by definition they are zero in $\NC'$. So the functor   $\FC_1$ factors through the ideal generated by all $I$-sequences,  giving a functor $\FC_2: \NC^{\oplus,0}_{\mathrm{BS}} \to \NC'$. The category $\NC'$ is  idempotent complete, so the functor $\FC_2$ lifts to a functor between the corresponding Karoubian completions $\FC_3: \NC \to \NC'$.

We will now prove that $\FC_3$  is an equivalence of categories by
finding an inverse equivalence $\GC_3: \NC' \to \NC$ . Let $\GC_1:
\HC^{\oplus,0}_{\mathrm{BS}}  \to \NC^{\oplus,0}_{\mathrm{BS}} $ be the  lift to the graded envelope of the canonical
projection $ \HC_{\mathrm{BS}}  \to \NC_{\mathrm{BS}} $. The functor $\GC_1$ lifts to a functor between the corresponding
Karoubian completions $\GC_2: \HC  \to \NC. $ This functor is 
zero on any $B_x\in \HC$ such that $x\notin ^I\hspace{-0.13cm}W$ because
any such element is a summand of an $I$-sequence. This gives us a
functor $\GC_3: \NC'  \to \NC $ that is clearly an inverse equivalence
to $\FC_3.$ 
 \end{proof}

\subsection{$_Q\hspace{-0.00cm}\NC$: a localization of $\NC$.}  \label{sec:coinv}
We
will see in this section that a certain localized version of  $\NC$ is very simple. Thus the situation  for $\NC$ is similar (in terms of simplicity) to that of
$\HC$ (see \cite{EW}). This result was  unexpected (at least to the
authors).

For $I\subseteq S$, define the
 ring $R_I:= R/\langle \alpha _s \vert s\in I
\rangle$. It is the largest quotient on which the parabolic group
$W_I$ acts trivially.
If $A$ is either the ring $R$ or the ring $R_I,$ we use the notation $A(\frac{1}{\Phi_I^c})$ for the
localization of $A$ by all the roots $\alpha \in \Phi$ that are not in $\Phi_I$. In
formulas, $A(\frac{1}{\Phi_I^c})=A[\alpha^{-1} \vert\alpha \in\ \Phi\  \mathrm{and}\ \alpha
\notin\Phi_I]$. We define $Q_I:= R_{I}(\frac{1}{\Phi_I^c})$
(i.e. ``kill $I$ and invert the rest"). Define the category $${}_Q\NC_{\mathrm{BS}}:=Q_I\otimes_{R_I}\NC_{\mathrm{BS}}.$$
This tensor product notation means that the objects of
${}_Q\NC_{\mathrm{BS}}$ are the same as the objects of
$\NC_{\mathrm{BS}}$ and
$\mathrm{Hom}_{{}_Q\NC_{\mathrm{BS}}}(X,Y):=Q_I\otimes_{R_I}\mathrm{Hom}_{\NC_{\mathrm{BS}}}(X,Y)$. We
remark that if $s\in I$ then $\alpha_s$ is zero in $\NC_{\mathrm{BS}}$
(because of the Barbell relation). That is  why $R_I$ acts on the left of $\mathrm{Hom}_{\NC_{\mathrm{BS}}}(X,Y)$.  Another remark is that $Q_I$ is ungraded, and so is the category ${}_Q\NC_{\mathrm{BS}}$. 
Finally, we define the object of study of the following section
${}_Q\NC:=({}_Q\NC_{\mathrm{BS}})^e$.

 The right action of $\HC_{\mathrm{BS}}$ on $\NC_{\mathrm{BS}}$ extends in the obvious way to a right action of $\HC_{\mathrm{BS}}$ on $Q_I\otimes_{R_I}\NC_{\mathrm{BS}}$ (it is easy to check that this is indeed an action, i.e. to check the coherence conditions). Then, if a monoidal category acts on some category, its idempotent completion acts on the  idempotent completion of the category. Thus, the category ${}_Q\NC$ is  a right $\HC$-module.

\begin{notation}\label{cc}
 When the context is clear, we will denote  the identity morphism  $\mathrm{id}_M: M\to M$,  just by $M$. 
\end{notation}

 The following theorem will be proved in the next section.

\begin{thm}\label{Q}
 In ${}_Q\NC$ there is a set of objects $\{K_x\}_{x\in ^I\hspace{-0.03cm}W}$ satisfying the following properties.
\begin{enumerate}
\item\label{4a} $K_{\mathrm{id}}=N_{\mathrm{id}}$ (the image  in ${}_Q\NC$ of the empty sequence in $ \HC_{\mathrm{BS}}$).
\item\label{4b} $K_xf=x(f)K_x\ \mathrm{for}\  f\in R$.
\item\label{4c} For all $x\in ^I\hspace{-0.13cm}W$ we have $K_xB_s
  \cong 
\begin{cases}
K_x\oplus K_{xs}\ \mathrm{if}\ xs\in ^I\hspace{-0.13cm}W,\\
\hspace{0.6cm} 0 \hspace{0.7cm} \ \mathrm{if} \ xs\notin ^I\hspace{-0.13cm}W.
\end{cases}
$
\item\label{4d} For all $x, y \in ^I\hspace{-0.13cm}W$ we have $\mathrm{Hom}(K_x, K_y)=
\delta_{x,y} Q_I\cdot \mathrm{id}_{K_x} $ (where $\delta_{x,y}$ is the Kronecker  delta). 
\item\label{4e} Any object in ${}_Q\NC$ is isomorphic to a direct sum of $K_x$'s.
\end{enumerate}
\end{thm}

\begin{remark}\label{nonzero}
In particular, part $(4)$ of this Theorem  implies that  $0\neq K_x\in {}_Q\NC.$
\end{remark}

\begin{remark} \label{rem:locHaction}
For a subset $J \subset S$ one can consider the localisation
$R[J^{-1}] := R[\alpha^{-1}]_{\alpha \in \Phi_J}$ and the corresponding
category of localised diagrammatic Soergel bimodules $\HC[J^{-1}] :=
R[J^{-1}] \otimes_R \HC$. (The case when $R$ is the fraction field is
discussed in \cite[\S 1.6]{EW}. Partial localisations also
make sense.) A fundamental aspect of the current article is that
$\HC[J^{-1}]$ usually does \emph{not} act on ${}_Q\NC$. For example,
if $s \in I \cap J \ne \emptyset$ then $\alpha_s$ is an invertible morphism in
$\HC[J^{-1}]$ but acts as zero on $K_{\mathrm{id}} \in
{}_Q\NC$. (A more mundane way of seeing that $\HC[S^{-1}]$ cannot act on ${}_Q\NC$
follows by observing that some of the structure constants in the standard bases for the action
of the Coxeter group on the anti-spherical module are negative.)

However, there is one situation where part of the localised category
does act. Suppose that $x \in ^I\hspace{-0.13cm}W$ and $J \subset S$
satisfies $xW_J \subset ^I\hspace{-0.13cm}W$. One can deduce, using the Parabolic Property, that $x(\alpha) \notin
\Phi_I$ for all $\alpha \in \Phi_J$. Thus right action by
$\alpha$ is invertible on $K_x \in {}_Q\NC$ for all $\alpha \in
\Phi_J$.  By the universal property of localization\footnote{Recall that another definition of an action $\star: \mathcal{M}\times A\rightarrow A$ of a monoidal category $\mathcal{M}$ on a category $A$ is a strong monoidal functor $F:\mathcal{M} \rightarrow \mathrm{End}(A)$ into the monoidal category of endofunctors of $A$.},  $\HC[J^{-1}]$ acts on the full subcategory generated by
$K_x$. In this case, for any $u \in W_J$ one
has a canonical isomorphism $K_x \cdot Q_u = K_{xu}$, where $Q_u
\in \HC[J^{-1}]$ denotes the object considered in \cite[\S 5.4]{EW}. 
\end{remark}


\section{Proof of Theorem \ref{Q}} \label{sec:proofs}

\begin{prop}
  $\End_{\NC}(D_{\id}) = R_I$.
\end{prop}

\begin{proof}
  By  definition of the category $\NC$, we have
  \[ \End_{\NC}(D_{\id}) =\End_{\HC}(B_{\id})/J, \] where $J$ is the ideal of
$\End_{\HC}(B_{\id})$ generated as an $R$-module by maps that factor
through an $I$-sequence.  

By the double leaves theorem in $\HC$ we know that $\End_{\HC}(B_{\id})=R$ (the only double leaf in $\End_{\HC}(B_{\id})$ is the identity). So, if one defines the ideal $$\alpha_I:=\langle \alpha_s\, : \, s\in I\rangle \subset R,$$  to finish the proof we just need to prove that $J=\alpha_I.$

It is easy to see that $J\supset \alpha_I$, because if $s\in I$,  $\alpha_s\in \End_{\HC}(R)$ can be factored through $B_s$: 
\begin{equation*}
 R \stackrel{\tikz[scale=.3]{ \draw[red] (0,1) to (0,.4);
    \node[circle,fill,draw,inner sep=0mm,minimum size=1mm,color=red] at (0,.4) {};
}}{\longto}  B_s \stackrel{\tikz[scale=-.3]{ \draw[red] (0,1) to (0,.4);
    \node[circle,fill,draw,inner sep=0mm,minimum size=1mm,color=red] at (0,.4) {};
}}{\longto}  R  
\end{equation*}

 Let us prove that $J\subset \alpha_I$. Any map in $f\in J$ can be written as 
 $$ f=\sum_{\substack{\underline{x}\ \mathrm{is\ an} \\
I-\mathrm{sequence}}}p_{\underline{x}}\, (g_{\underline{x}}\circ h_{\underline{x}}),$$ 
where $p_{\underline{x}}$ is an element of $R$, $g_{\underline{x}}: \underline{x}\rightarrow B_{\id}$  and $h_{\underline{x}}: B_{\id}\rightarrow \underline{x}. $ By the double leaves theorem, each  $g_{\underline{x}}$   can be written as an $R$-linear combination of light leaves. We remark that we don't mean double leaves but honest light leaves, given that if the  codomain is $B_{\id}$,  double leaves are light leaves.  The same can be said of $h_{\underline{x}}$ (with upside-down light leaves). 

 Thus it is enough to prove that if $ f=g_{\underline{x}}\circ h_{\underline{x}},$ with $\underline{x}$ an $I$-sequence, $g_{\underline{x}}$ a light leaf and  $h_{\underline{x}}$ an upside-down light leaf, then $f\in \alpha_I.$ We prove this by induction on the length of $\underline{x}$. 

Suppose that $\underline{x}$ has $s$ in its left-most position. If
$g_{\underline{x}}$ and $h_{\underline{x}}$ have a $U0$  in the
left-most position then $f\in \langle \alpha_s \rangle$ and we are
done.

Suppose that either $g_{\underline{x}}$ or $h_{\underline{x}}$ have a $U0$ that is not in the left-most position. To fix ideas, say that it is the case for $g_{\underline{x}}$. Then one can write $g_{\underline{x}}$ as the composition of a dot and  $g_{\underline{y}}: \underline{y}\rightarrow B_{\id}$, with $\underline{y}$ an $I$-sequence and
$l(\underline{y})<l(\underline{x})$ as in the picture:
$$
\begin{tikzpicture}[x=0.5pt,y=0.5pt,yscale=-1,xscale=1]
\draw    (319.5,44) -- (267.5,111) ;
\draw    (319.5,44) -- (367.5,110) ;
\draw    (267.5,111) -- (367.5,110) ;
\draw    (368.2,139.8) -- (267,140.6) ;
\draw    (298.6,129) -- (298.4,140.2) ;
\draw   (296.32,126.87) .. controls (296.27,125.65) and (297.22,124.63) .. (298.44,124.59) .. controls (299.66,124.54) and (300.68,125.5) .. (300.73,126.72) .. controls (300.77,127.93) and (299.82,128.96) .. (298.6,129) .. controls (297.38,129.04) and (296.36,128.09) .. (296.32,126.87) -- cycle ;
\draw   (297.8,126.82) .. controls (297.78,126.42) and (298.1,126.08) .. (298.5,126.07) .. controls (298.9,126.06) and (299.23,126.37) .. (299.25,126.77) .. controls (299.26,127.17) and (298.95,127.5) .. (298.55,127.52) .. controls (298.15,127.53) and (297.81,127.22) .. (297.8,126.82) -- cycle ;
\draw   (297.27,126.84) .. controls (297.25,126.15) and (297.79,125.57) .. (298.48,125.55) .. controls (299.17,125.52) and (299.74,126.06) .. (299.77,126.75) .. controls (299.79,127.44) and (299.25,128.02) .. (298.57,128.04) .. controls (297.88,128.07) and (297.3,127.53) .. (297.27,126.84) -- cycle ;
\draw   (300.14,127.6) .. controls (300.12,127.12) and (299.72,126.75) .. (299.25,126.77) .. controls (298.77,126.79) and (298.4,127.18) .. (298.42,127.66) .. controls (298.43,128.14) and (298.83,128.51) .. (299.31,128.49) .. controls (299.78,128.47) and (300.15,128.07) .. (300.14,127.6) -- cycle ;
\draw    (267,140.6) -- (321,203) ;
\draw    (368.2,139.8) -- (321,203) ;
\draw   (385,138.2) .. controls (389.67,138.18) and (391.99,135.84) .. (391.97,131.17) -- (391.84,101.17) .. controls (391.81,94.5) and (394.13,91.16) .. (398.8,91.14) .. controls (394.13,91.16) and (391.79,87.84) .. (391.76,81.17)(391.77,84.17) -- (391.64,51.17) .. controls (391.62,46.5) and (389.28,44.18) .. (384.61,44.2) ;
\draw (309,75.73) node [anchor=north west][inner sep=0.75pt]    {$g_{\underline{y}}$};
\draw (310.33,150.4) node [anchor=north west][inner sep=0.75pt]    {$h_{\underline{x}}$};
\draw (185.33,116.4) node [anchor=north west][inner sep=0.75pt]    {$f\ =$};
\draw (405,78.93) node [anchor=north west][inner sep=0.75pt]    {$g_{\underline{x}}$};
\end{tikzpicture}
$$
By the double leaves theorem, the part of the diagram that is below $g_{\underline{y}}$ (i.e. the composition of $h_{\underline{x}}$ and the dot) can be written as an $R$-linear combination of leaves $\underline{y}\rightarrow B_{\id}$ flipped upside-down, so by the induction hypothesis we are done.  

So we are left with two cases:  the first case is that the light leaves of $f$ (i.e. $g_{\underline{x}}$ and $h_{\underline{x}}$) don't have $U0$'s. The second case is that one of them has one $U0$ in the left-most position and the other one has no $U0$'s.


The last step of any light leaf with codomain $B_{\id}$ can only be  $U0$ or  $D1$ ($D0$ and $U1$ don't produce the correct codomain). But in our cases the last step can not be a $U0$, so we conclude that in both cases, the last step of both light leaves of $f$
  are $D1$'s. Thus we have: 
$$
\begin{tikzpicture}[x=0.5pt,y=0.5pt,yscale=-1,xscale=1]
\draw    (179.7,70.6) -- (141.7,119.6) ;
\draw    (250.2,120.4) -- (141.7,119.6) ;
\draw    (209,70) -- (250.2,120.4) ;
\draw    (209,70) -- (179.7,70.6) ;
\draw    (180.2,169.6) -- (141.7,119.6) ;
\draw    (250.2,120.4) -- (210.2,170.4) ;
\draw    (210.2,170.4) -- (180.2,169.6) ;
\draw  [draw opacity=0] (184.66,70.84) .. controls (184.65,70.66) and (184.65,70.48) .. (184.65,70.3) .. controls (184.65,63.18) and (188.99,57.4) .. (194.35,57.4) .. controls (199.52,57.4) and (203.74,62.77) .. (204.03,69.54) -- (194.35,70.3) -- cycle ; \draw   (184.66,70.84) .. controls (184.65,70.66) and (184.65,70.48) .. (184.65,70.3) .. controls (184.65,63.18) and (188.99,57.4) .. (194.35,57.4) .. controls (199.52,57.4) and (203.74,62.77) .. (204.03,69.54) ;
\draw  [draw opacity=0] (205.2,170.3) .. controls (205.22,170.64) and (205.23,170.99) .. (205.23,171.33) .. controls (205.23,178.46) and (200.89,184.23) .. (195.53,184.23) .. controls (190.18,184.23) and (185.83,178.46) .. (185.83,171.33) .. controls (185.83,170.9) and (185.85,170.47) .. (185.88,170.05) -- (195.53,171.33) -- cycle ; \draw   (205.2,170.3) .. controls (205.22,170.64) and (205.23,170.99) .. (205.23,171.33) .. controls (205.23,178.46) and (200.89,184.23) .. (195.53,184.23) .. controls (190.18,184.23) and (185.83,178.46) .. (185.83,171.33) .. controls (185.83,170.9) and (185.85,170.47) .. (185.88,170.05) ;
\draw    (353.03,69.27) -- (315.03,118.27) ;
\draw    (423.53,119.07) -- (315.03,118.27) ;
\draw    (382.33,68.67) -- (423.53,119.07) ;
\draw    (382.33,68.67) -- (353.03,69.27) ;
\draw    (353.53,168.27) -- (315.03,118.27) ;
\draw    (423.53,119.07) -- (383.53,169.07) ;
\draw    (383.53,169.07) -- (353.53,168.27) ;
\draw  [draw opacity=0] (357.99,69.51) .. controls (357.99,69.33) and (357.98,69.15) .. (357.98,68.97) .. controls (357.98,61.84) and (362.33,56.07) .. (367.68,56.07) .. controls (372.85,56.07) and (377.07,61.44) .. (377.37,68.21) -- (367.68,68.97) -- cycle ; \draw   (357.99,69.51) .. controls (357.99,69.33) and (357.98,69.15) .. (357.98,68.97) .. controls (357.98,61.84) and (362.33,56.07) .. (367.68,56.07) .. controls (372.85,56.07) and (377.07,61.44) .. (377.37,68.21) ;
\draw  [draw opacity=0] (378.54,168.97) .. controls (378.56,169.31) and (378.57,169.65) .. (378.57,170) .. controls (378.57,177.12) and (374.22,182.9) .. (368.87,182.9) .. controls (363.51,182.9) and (359.17,177.12) .. (359.17,170) .. controls (359.17,169.57) and (359.18,169.14) .. (359.21,168.71) -- (368.87,170) -- cycle ; \draw   (378.54,168.97) .. controls (378.56,169.31) and (378.57,169.65) .. (378.57,170) .. controls (378.57,177.12) and (374.22,182.9) .. (368.87,182.9) .. controls (363.51,182.9) and (359.17,177.12) .. (359.17,170) .. controls (359.17,169.57) and (359.18,169.14) .. (359.21,168.71) ;
\draw    (367.67,56.17) -- (367.67,44.83) ;
\draw    (368.67,194.5) -- (368.67,183.17) ;
\draw   (365.83,42.92) .. controls (365.83,41.86) and (366.65,41) .. (367.67,41) .. controls (368.68,41) and (369.5,41.86) .. (369.5,42.92) .. controls (369.5,43.98) and (368.68,44.83) .. (367.67,44.83) .. controls (366.65,44.83) and (365.83,43.98) .. (365.83,42.92) -- cycle ;
\draw   (366.83,196.42) .. controls (366.83,195.36) and (367.65,194.5) .. (368.67,194.5) .. controls (369.68,194.5) and (370.5,195.36) .. (370.5,196.42) .. controls (370.5,197.48) and (369.68,198.33) .. (368.67,198.33) .. controls (367.65,198.33) and (366.83,197.48) .. (366.83,196.42) -- cycle ;
\draw  [color={rgb, 255:red, 0; green, 0; blue, 0 }  ][line width=3] [line join = round][line cap = round] (367.6,43.5) .. controls (367.67,43.3) and (367.73,43.1) .. (367.8,42.9) ;
\draw  [color={rgb, 255:red, 0; green, 0; blue, 0 }  ][line width=3] [line join = round][line cap = round] (368.4,196.7) .. controls (368.47,196.63) and (368.53,196.57) .. (368.6,196.5) ;
\draw  [dash pattern={on 4.5pt off 4.5pt}]  (298.71,50.29) -- (454.33,50.67) ;
\draw  [dash pattern={on 4.5pt off 4.5pt}]  (299,188) -- (450.33,189.33) ;
\draw (89.33,110.73) node [anchor=north west][inner sep=0.75pt]    {$f\ =$};
\draw (260.67,110.07) node [anchor=north west][inner sep=0.75pt]    {$\ \ =$};
\end{tikzpicture}
$$
The second equality is by the definition of the cup and cap. The map between the dotted lines is a negative degree map (degree $-2$) in $\mathrm{Hom}(B_s, B_t)$, where $s$ can be the same as $t$. This implies that $f=0$.

\end{proof}

\begin{cor}
$\End_{{}_Q\NC}(B_{\id}) = Q_I$.  
\end{cor}

We now turn to the proof of Theorem \ref{Q}. It relies on a modest amount of homological
algebra. Given an additive category, $\AC$ we denote by $K^b(\AC)$ the
homotopy category of bounded complexes in $\AC$. It is a triangulated
category. If $\AC$ is in
addition monoidal, then $K^b(\AC)$ is monoidal under tensor product of
complexes. If $\MC$ is a right $\AC$-module, then $K^b(\MC)$ is a
right $K^b(\AC)$-module.

In particular, $K^b(\HC)$ is a monoidal category, and it has right
modules $K^b(\NC)$ and $K^b({}_Q\NC)$. An important role will be
played by Rouquier complexes. For any $s \in S$ consider the complex
\[
F_s := 0 \to  B_s \stackrel{\tikz[scale=-.4]{ \draw[red] (0,1) to (0,.4);
    \node[circle,fill,draw,inner sep=0mm,minimum size=1mm,color=red] at (0,.4) {};
}}{\longto} R(1) \to 0
\]
where $B_s$ is in degree 0. It is known that $F_s$ is an invertible element of $K^b(\HC)$, with
inverse
\[
F_s^{-1} := 0 \to R(-1) \stackrel{\tikz[scale=.4]{ \draw[red] (0,1) to (0,.4);
    \node[circle,fill,draw,inner sep=0mm,minimum size=1mm,color=red] at (0,.4) {};
}
}{\longto} B_s \to 0 
\]
where $B_s$ is again in degree zero. Moreover, given any element $w
\in W$ we set
\[
F_w = F_{\un{w}} := F_sF_t \cdots F_u,
\]
where $\un{w} := st \cdots u$ is a reduced expression for $w$. This
complex does not depend on the reduced expression chosen. The
above results are due to Rouquier \cite{Ro}\footnote{In fact, \cite{Ro} shows
  that $F_w$ is defined up to canonical isomorphism. We won't need this
  stronger statement below.}. The reader may consult
\cite{AMRW2} for an in-depth discussion of Rouquier complexes in
the diagrammatic language.

The following beautiful little lemma is apparently well-known in the link
homology literature (see e.g. \cite{gorsky2017hilbert}):

\begin{lem} \label{lem:fxf}
  We have $F_x \cdot f = x(f) \cdot F_x$ as endomorphisms of $F_x \in K^b(\HC)$.
\end{lem}

\begin{proof}
  It is enough to check this for $x = s \in S$ a simple reflection and $f$ a homogeneous polynomial. 
   In  this case we need to check that $s(f) \cdot F_s- F_s \cdot f$ is
  null-homotopic. This map of complexes is
  \[ \begin{tikzpicture}[xscale=1.2,yscale=0.7]
\node (ul)  at (-1,1) {$B_s$};
\node (ur) at (1,1) {$R(1)$};
\node (ll) at (-1,-1) {$B_s(d)$};
\node (lr) at (1,-1) {$R(d+1)$};
\draw [->] (ul) to (ur);
\draw [->] (ll) to (lr);
\draw[->] (ul) to node[left] {$a_0$} (ll);
\draw[->] (ur) to node[right] {$a_1$} (lr);
\end{tikzpicture}
\]
where $d = \deg f$ and
\begin{gather*}
  a_0 = \begin{array}{c}
\tikz[scale=.7]{ \draw[dashed] (-1,1) rectangle (1,-1); \draw[red]
          (0,1) to (0,-1); \node at (-.5,0) {\small $s(f)$}; }
        \end{array}
        -
        \begin{array}{c}
\tikz[scale=.7]{ \draw[dashed] (-1,1) rectangle (1,-1); \draw[red]
          (0,1) to (0,-1); \node at (.5,0) {\small $f$}; }
        \end{array}
= \begin{array}{c}
\tikz[scale=.7]{ \draw[dashed] (-1,1) rectangle (1,-1); \draw[red]
    (0,1) to (0,.4);
    \node[circle,fill,draw,inner sep=0mm,minimum size=1mm,color=red] at (0,.4) {};
    \node[circle,fill,draw,inner sep=0mm,minimum size=1mm,color=red] at (0,-.4) {};
\draw[red] (0,-.4) to (0,-1); 
    \node at (0,0) {\small $-\partial_sf$}; }
  \end{array}
  \end{gather*}
  and
  \begin{gather*}
          a_1 = \begin{array}{c}
\tikz[scale=.7]{ \draw[dashed] (-1,1) rectangle (1,-1); \node at (0,0)
                  {\small $s(f) -f$ }; }
        \end{array}.
      \end{gather*}
      Now one checks directly that
      \[
h = \begin{array}{c}
  \tikz[scale=.7]{ \draw[dashed] (-1,1) rectangle (1,-1); \draw[red]
    (0,1) to (0,.4);
    \node[circle,fill,draw,inner sep=0mm,minimum size=1mm,color=red] at (0,.4) {};
    \node at (0,0) {\small $-\partial_sf$}; }
  \end{array} : R(1) \to B_s(d).
\]
provides the null-homotopy.
\end{proof}

We now turn to the proof in earnest.  Define $K_{\id}$ to be the image  in ${}_Q\NC$ of $D_{\id}$, also known as the empty sequence in $ \HC_{\mathrm{BS}}$.

\begin{prop} \label{prop:Kx}
  If $x \in {}^I W$ then $K_{\id} \cdot F_x$ is isomorphic to a
  complex concentrated in degree zero. Moreover, we have an isomorphism
\[
K_{\id} \cdot F_x \cong K_{\id} \cdot (F_{x^{-1}})^{-1}
\]
\end{prop}

In other words, once we have proved the proposition we know that there
exist objects
\[
  K_x \in {}_Q \NC \quad \text{for each $x \in {}^I W$}
\]
  such that
\[
K_x \cong K_{\id} \cdot F_x \in K^b({}_Q\NC).
\]
These objects will play a key role in the proof of Theorem
\ref{Q}.

\begin{proof}
  We will prove the proposition by induction on the length of $x$,
  with both statements in case $\ell(x) = 0$ being trivial. Write $x = ys$ with
  $\ell(x) = \ell(y) + 1$ and $y \in {}^IW$. By induction, there exists $K_{y} \in {}_Q \NC$
  such that
  \[
K_{y} \cong K_{\id} \cdot F_{y} \cong K_{\id} \cdot (F_{y^{-1}})^{-1} \in K^b({}_Q\NC).
\]
In particular, $K_{\id} \cdot F_{x}  =  K_{\id} \cdot F_{y}F_s$ is
isomorphic to the two-term complex
\begin{equation}
  \label{eq:2termcomplex}
\dots \to 0 \to K_{y}B_s \stackrel{K_y\tikz[scale=-.3]{ \draw[red] (0,1) to (0,.4);
    \node[circle,fill,draw,inner sep=0mm,minimum size=1mm,color=red] at (0,.4) {};
}}{\longto} K_{y}R \to \dots  
\end{equation}
whilst $K_{\id} \cdot (F_{x^{-1}})^{-1}  =  K_{\id} \cdot (F_{y^{-1}})^{-1}F_s^{-1}$ is
isomorphic to the two-term complex
\begin{equation}
  \label{eq:2termcomplex2}
\dots \to K_{y}R \stackrel{K_y\tikz[scale=.3]{ \draw[red] (0,1) to (0,.4);
    \node[circle,fill,draw,inner sep=0mm,minimum size=1mm,color=red] at (0,.4) {};
}}{\longto} K_{y}B_s \to 0 \to \dots  
\end{equation}
with $K_{y}B_s$ in degree zero in both complexes. 
 The composition
\begin{equation}
  \label{eq:comp}
K_{y} R \stackrel{K_y\tikz[scale=.3]{ \draw[red] (0,1) to (0,.4);
    \node[circle,fill,draw,inner sep=0mm,minimum size=1mm,color=red] at (0,.4) {};
}}{\longto} K_{y} B_s \stackrel{K_y\tikz[scale=-.3]{ \draw[red] (0,1) to (0,.4);
    \node[circle,fill,draw,inner sep=0mm,minimum size=1mm,color=red] at (0,.4) {};
}}{\longto} K_{y} R  
\end{equation}
is equal to $K_{y} \alpha_s$. By Lemma \ref{lem:fxf}, we have
\[
K_{y} \alpha_s = K_{\id} \cdot F_{y} \alpha_s = y(\alpha_s) \cdot F_{y}
\]
which is invertible by the parabolic property \ref{pp}. In particular we can find an isomorphism
\[
K_yB_s \cong K_y \oplus X
\]
such that the differentials in \eqref{eq:2termcomplex}
(resp. \eqref{eq:2termcomplex2}) are (up to a scalar) the projection (resp. inclusion)
of $K_y = K_yR$. Removing this contractible summand, we deduce
that
  \[
X \cong K_{\id} \cdot F_{x} \cong K_{\id} \cdot (F_{x^{-1}})^{-1} \in K^b({}_Q\NC).
\]
and the proposition follows.
\end{proof}

The following details the behaviour of $K_x$ under $F_s$ in general. 

\begin{prop} \label{prop:action}
  For any $x \in {}^I W$ we have
  \begin{equation*}
    K_x \cdot F_s \cong \begin{cases}
K_{xs} & \text{if $xs \in {}^I W$,} \\
K_{x}[-1] & \text{if $xs \notin {}^I W$.} \end{cases}
  \end{equation*}
\end{prop}

\begin{remark}
We leave it to the reader to formulate an analogous result for
$(F_s)^{-1}$. In particular, the $K_x$ are preserved under the action
of the braid group. (We will not need this fact below).
\end{remark}

\begin{proof}
  Let us first assume $xs \in {}^I W$. The only case not already directly covered by Proposition \ref{prop:Kx} is when $xs < x$. But then
\[
K_x \cdot F_s \cong K_{\id} \cdot (F_{x^{-1}})^{-1} F_s \cong  K_{\id} \cdot (F_{(xs)^{-1}})^{-1} \cong K_{xs}.
\]
by Proposition \ref{prop:Kx}.

We now examine the second case. If $x \in {}^IW$ and $xs \notin {}^IW$
then $xs = tx$ for some $t \in I$  (proof in Section \ref{pp}).  We
have 
\[
K_xF_s \cong K_{\id} F_{xs} \cong K_{\id} F_t F_x \stackrel{(*)}{\cong} K_{\id}[-1] F_x \cong K_x[-1]
\]
where for $(*)$ we use that
\[
D_{\id} ( B_t \to R) = (0 \to D_{\id}) = D_{\id}[-1]
\]
because $B_t = 0$ in $\NC$. (We ignore internal  shifts (i.e. $(1)$'s), as they
don't affect the outcome.)
\end{proof}

\begin{prop} \label{prop:homid}
  For $\id \ne x \in  {}^IW$ we have
  \[
\Hom_{{}_Q\NC}(K_x, K_{\id}) = 0.
\]
\end{prop}

\begin{proof} Let us choose a reduced expression $\un{x} = (s_1, \dots,
  x_m)$ for $x$. Consider the tensor product of complexes
  \[
F_{\un{x}} = F_{s_1}F_{s_2} \dots F_{s_m} = ( \dots \to 0 \to B_{\un{x}}
\stackrel{d_0}{\longto} \bigoplus_{i = 1}^m B_{\un{x}_{\hat{i}}} \to \dots)
\]
where:
\begin{enumerate}
\item $B_{\un{x}}$ is in degree 0;
 \item $B_{\un{x}_{\hat{i}}}$ indicates the tensor product of $B_{s_1}
   B_{s_2} \dots B_{s_m}$ with $B_{s_i}$ omitted; 
 \item the differential $d_0$ is a direct sum of dot maps tensored with copies of the identity map (up to sign).
 \end{enumerate}
 We first claim that
 \[
\Hom_{K^b(\HC)}(F_{\un{x}}, R) = 0.
\]
This is a special case of the main result in \cite{LiWi}.  But this particular case is easy enough that can be proved directly. Unpacking the definitions, the equality holds if and only if
\[
\Hom_{\HC}(\bigoplus_{i = 1}^m B_{\un{x}_{\hat{i}}}, R) \to \Hom_{\HC}(B_{\un{x}},R)
\]
is surjective. However, this is the case by the light leaves theorem, because any map in $ \Hom_{\HC}(B_{\un{x}},R)$ is an $R$-linear combination of light leaves, and each of these must have a $U0$ (i.e.  a dot) after a certain number of $U1'$s,  because $x\neq \mathrm{id}$.

Now consider the commutative diagram:
\[
  \begin{tikzpicture}[xscale=1.2,yscale=.7]
\node (ul) at (-2,2) {$\Hom_{\HC}(\bigoplus_{i = 1}^m
  B_{\un{x}_{\hat{i}}}, R)$};
\node (ur) at (2,2) {$\Hom_{\HC}(B_{\un{x}},R)$};
\node (ml) at (-2,0) {$\Hom_{\NC}(\bigoplus_{i = 1}^m
  B_{\un{x}_{\hat{i}}}, R)$};
\node (mr) at (2,0) {$\Hom_{\NC}(B_{\un{x}},R)$};
\node (ll) at (-2,-2) {$Q_I \otimes \Hom_{\NC}(\bigoplus_{i = 1}^m
  B_{\un{x}_{\hat{i}}}, R)$};
\node (lr) at (2,-2) {$Q_I \otimes \Hom_{\NC}(B_{\un{x}},R)$};
\draw[->] (ul) to node[above] {$(\HC)$} (ur);
\draw[->] (ml) to node[above] {$(\NC)$} (mr);
\draw[->] (ll) to node[above] {$({}_Q\NC)$} (lr);
\draw[->] (ul) to (ml);
\draw[->] (ml) to (ll);
\draw[->] (ur) to (mr);
\draw[->] (mr) to (lr); 
\end{tikzpicture}
  \]
We have just argued that the arrow labelled $(\HC)$ is surjective,
hence so is $(\NC)$ (the upper vertical maps are surjections by
definition of the morphisms in a quotient category), and hence so is
the arrow labelled $({}_Q\NC)$ (as $Q_I \otimes(-)$ preserves
surjections). This implies that
\[
  \Hom_{{}_Q\NC}(K_{x}, K_{\id}) = 
\Hom_{K^b({}_Q\NC)}(K_{x}, K_{\id}) = 
\Hom_{K^b({}_Q\NC)}(K_{\id} F_{\un{x}}, K_{\id}) = 0
\]
as claimed.
\end{proof}

\begin{remark}
  The objects $D_{\id} \cdot F_x$ and $D_{\id} \cdot
  F_{x^{-1}}^{-1}$ in $K^b(\NC)$ (for $x \in {}^IW$) should satisfy vanishing
  conditions generalizing the well-known vanishing

  \[\dim \Ext^i(\Delta_\lambda,
  \nabla_\mu) = \delta_{0,i}\delta_{\lambda,\mu}\]in highest weight
  categories. For the homotopy category of the Hecke category, this is
  proved in \cite{LiWi}. It is likely that the techniques of 
  \cite{MackyMoment,AR1,AR2} are most easily
  generalized to this setting. Such a vanishing result in $K^b(\NC)$
  would imply most results in this section.
\end{remark}

\begin{prop} \label{prop:homKxKy}
For $x, y \in {}^IW$ we have
  \[
    \Hom(K_x,K_y) =
     \begin{cases}
 Q_I & \text{if $x = y$,} \\
 0 & \text{otherwise.}
 \end{cases}
    \]
  \end{prop}

  \begin{proof} Action by the equivalence $F_x$ gives us identifications
    \[
Q_I = \End_{{}_Q\NC}(K_{\id}) = \End_{K^b({}_Q\NC)} (K_{\id} F_x) = \End_{{}_Q\NC}(K_x)
\]
from which the first statement follows.

We now proceed to the vanishing statement. By Proposition
\ref{prop:Kx} we have
\[
\Hom_{K^b({}_Q\NC)} (K_x, K_y) = \Hom_{K^b({}_Q\NC)}(K_x, K_{\id}
(F_{y^{-1}})^{-1}) =
\Hom_{K^b({}_Q\NC)}(K_xF_{y^{-1}}, K_{\id})
  \]
By Proposition \ref{prop:action} we know that $K_xF_{y^{-1}}$ is
isomorphic to $K_z[m]$ for some $z \in {}^I W$ and $m \le 0$, with $z
= {\id}$ and $m = 0$ if and only if $x = y$.  If $ m < 0$ we are done,
as there are no maps between complexes concentrated in different
degrees. If $m = 0$ then we are done by Proposition \ref{prop:homid}.
\end{proof}

Finally we establish:
\begin{prop} \label{prop:Bs}
  For $x \in {}^IW$ we have
\[ K_xB_s
  \cong 
\begin{cases}
K_x\oplus K_{xs}\ \mathrm{if}\ xs\in ^I\hspace{-0.13cm}W,\\
\hspace{0.6cm} 0 \hspace{0.7cm} \ \mathrm{if} \ xs\notin ^I\hspace{-0.13cm}W.
\end{cases}
\]
\end{prop}

\begin{proof}
  The fact that $K_xB_s \cong K_x \oplus K_{xs}$ when $xs > x$ and $xs
  \in {}^IW$ follows
  from the proof of Proposition \ref{prop:Kx}.  We now consider the
  case when $xs < x$ (and then necessarily $xs \in {}^IW$). The stupid
  filtration on $F_s^{-1}$  yield a distringuished triangle
  \[
B_s \to F_s^{-1} \to R(-1)[1] \triright
\]
This distringuished triangle can also be obtained as the mapping cone on $B_s\rightarrow F_s^{-1}$. If we turn it, we obtain
 \[
R(-1) \to B_s \to F_s^{-1} \triright.
\]
If we act on $K_x$ with this triangle we obtain a distinguished triangle
\[
K_x \to  K_xB_s \to K_{xs} \triright
\]
and hence $K_xB_s \cong K_x \oplus K_{xs}$ because $\Hom(K_{xs},
K_x[1]) = 0$.

We now consider the case when $xs \notin {}^IW$. Then necessarily $xs
= tx$ for some $t \in I$. We claim that in this case we have
isomorphisms
\begin{equation} \label{eq:xs=tx}
F_xB_s \cong B_tF_x \quad \text{in $K^b(\HC)$.}
\end{equation}
To establish \eqref{eq:xs=tx}, first note that $\Hom(R(1)[-1], F_s)$
(degree zero morphisms) is one-dimensional. Hence left and right
action by the equivalences $F_x$ allows us
to deduce the same statement for $\Hom(F_x(1)[-1], F_xF_s)$ and
$\Hom(F_x(1)[-1], F_tF_x)$. In particular, we have a commutative
diagram
\begin{equation}
  \label{eq:2}
  \begin{array}{c}
  \begin{tikzpicture}[yscale=.8]
\node (ul) at (-2,1) {$F_x(1)[-1]$};
\node (um) at (0,1) {$F_xF_s$};
\node (ur) at (2,1) {$F_xB_s$};
\node (ll) at (-2,-1) {$F_x(1)[-1]$};
\node (lm) at (0,-1) {$F_tF_x$};
\node (lr) at (2,-1) {$B_tF_x$};
\draw[->] (ul) to (um);
\draw[->] (um) to (ur);
\draw[->] (ll) to (lm);
\draw[->] (lm) to (lr);
\draw[->] (ul) to node[right] {$\sim$} (ll);
\draw[->] (um) to node[right] {$\sim$} (lm);
\draw[->] (ur) to node[above] {$[1]$} (3.5,1);
\draw[->] (lr) to node[above] {$[1]$} (3.5,-1);
\end{tikzpicture}
    \end{array}
  \end{equation}
 The triangles are obtained via action on the triangles
  \[
R(1)[-1] \to F_u \to B_u \triright
\]
by $F_x$ on the left (with $u = s$) and right (with $u = t$). Now
   \eqref{eq:xs=tx} follows by the existence of a (non-canonical)
   isomorphism of cones.

   We are done:
\[
K_xB_s = K_{\id} \cdot F_xB_s \stackrel{\eqref{eq:xs=tx}}{\cong} K_{\id} \cdot B_tF_x = 0
\]
because $B_t = 0$ in $\NC$.
\end{proof}

We now turn to the proof of the main theorem:

\begin{proof}[Proof of Theorem \ref{Q}] 
  (1) is immediate from the definitions, (2) follows from Lemma
  \ref{lem:fxf}, (3) follows from Proposition \ref{prop:Bs}, (4) follows from Proposition \ref{prop:homKxKy}, and
  (5) is immediate from (3).
\end{proof}

\section{$I$-antispherical double leaves are a basis}

In this section we will follow the notation of \cite[Construction
6.1]{EW}, where light leaves and double leaves are explained in
diagrammatic terms. However we make a slight modification of the construction therein. In the definition of $\phi_k$  start by doing the following. If $w_{k-1}s\notin ^I\hspace{-0.13cm}W$ and $\mathrm{e}_k$ is either $U0$ or $U1$ then apply some loop in the rex graph starting (and ending) in $\underline{w}_{k-1}s$ and passing through an $I$-sequenceAn.  If not, do nothing. This slight modification in the construction changes nothing in the proof that these morphisms give bases of the corresponding Hom spaces. 

Recall that if $\LL_{{\un{x}},\eb} \co B_{\un{x}} \to B_{\un{w}}$ is a light leaves map where ${\un{w}}$ is a rex for $w$, by flipping this diagram upside-down, we get a map $\oLL_{{\un{x}},\eb} \co
B_{\un{w}} \to B_{\un{x}}$. 

Let ${\un{x}}$ and ${\un{y}}$ be arbitrary sequences with subsequences $\eb$ and $\fb$ respectively, such that $({\un{x}},\eb)$ and $({\un{y}},\fb)$ both express $w$. Choose a rex ${\un{w}}$ for $w$, and
construct maps $\LL_{{\un{x}},\eb} \co B_{\un{x}} \to B_{\un{w}}$ and $\oLL_{{\un{y}},\fb} \co B_{\un{w}} \to B_{\un{y}}$. The corresponding \emph{double leaves map} is the composition 
\[
\LLL_{w,\fb,\eb} \, \define \, \,  \oLL_{{\un{y}},\fb} \circ \LL_{{\un{x}},\eb}.
\] 
Finally, $\mathbb{LL}_{\underline{x},\underline{y}}$ is the set consisting of all double leaves of the form $\LLL_{w,\fb,\eb}$ (for all $w\in W$ and all subexpressions $\eb$ and $\fb$ such that $({\un{x}},\eb)$ and $({\un{y}},\fb)$ both express $w$.)

\begin{defi} 
A subexpression $\eb$ of $\underline{x}=s_{1}s_{2}\cdots s_{m}$ ($\underline{x}$ is not necessarily reduced) is \emph{I-antispherical} if for all $0\leq k< m$ we have $$x_ks_{{k+1}}\in ^I\hspace{-0.13cm}W,
$$ where $x_i$ is the $i^{\mathrm{th}}$-\emph{point} of the Bruhat stroll. An element of the set $LL_{\underline{x}, \eb}$ is called an \emph{I-antispherical light leaf} if $\eb$ is an $I$-antispherical subexpression of  $\underline{x}$.  An \emph{I-antispherical double leaf} is a double leaf  which is a composition of two  $I$-antispherical light leaves. Let $\underline{y}$ be another expression. Then we define the set $\mathbb{LL}^{I\mathrm{as}}_{\underline{x},\underline{y}}$ as the subset of $\mathbb{LL}_{\underline{x},\underline{y}}$ consisting of $I$-antispherical double leaves. We will also denote by $\mathbb{LL}^{I\mathrm{as}}_{\underline{x},\underline{y}}$ this set in $\mathcal{N}.$

\end{defi}

\begin{remark}\label{bre}
Let $J\subseteq I\subseteq S$. As ${}^IW\subseteq {}^JW$, we have that if a light leaf is $I$-antispherical, then it is also $J$-antispherical.
\end{remark}

\begin{thm}\label{DLT}
Let  $\underline{x}$ and $\underline{y}$ be (not necessarily reduced) expressions.
 The set
$\mathbb{LL}^{I\mathrm{as}}_{\underline{x},\underline{y}}$  forms a
free $R_I$-basis for $\mathrm{Hom}^{\bullet}_{\mathcal{N}}(N_{\underline{x}},
N_{\underline{y}})$ as a left module. 
\end{thm}

\begin{remark}
In several cases (for instance, if $\underline{x}$ or $\underline{y}$ are $I$-sequences) this theorem just says that an empty set is a basis of the zero module. 
\end{remark}

\begin{proof}

In $\mathcal{H}$ the set $\mathbb{LL}_{\underline{x},\underline{y}}$ generates over $R$ (moreover is an $R$-basis for) the space $\mathrm{Hom}^{\bullet}_{\mathcal{H}}(B_{\underline{x}}, B_{\underline{y}})$. By definition of $\mathcal{N}$ we deduce that the set $\mathbb{LL}_{\underline{x},\underline{y}}$ generates over  $R$ the space $\mathrm{Hom}^{\bullet}_{\mathcal{N}}(N_{\underline{x}}, N_{\underline{y}})$. As $\alpha_s=0\in \mathcal{N}$ if $s\in I,$ we deduce that  the set $\mathbb{LL}_{\underline{x},\underline{y}}$ generates over  $R_I$ the space $\mathrm{Hom}^{\bullet}_{\mathcal{N}}(N_{\underline{x}}, N_{\underline{y}})$. 

But it is easy to see that with the slightly modified construction of
light leaves (discussed above)  a light leaf that is not $I$-antispherical    is zero in $\mathcal{N}$. Indeed, suppose there is some $0\leq k< m$ such that  $x_{k}s_{k+1}\notin {}^IW.$
Consider the least $k$ with that property. If $k=0$ then  $s_1\in I$ and the light leaf is zero. If $k>1$ then $x_{k-1}s_{k}\in {}^IW,$
thus in any case $x_{k}\in {}^IW.$
By Fact (c) in Section  \S \ref{ra}, we have the inequality
  $x_{k}<x_ks_{k+1}.$ So we have that $\mathrm{e}_{k+1}$ is either
  $U0$ or $U1$, so by the slight modification of the construction, the
  light leaf  factors through some  $I$-sequence, thus is zero. So we
  have proved that the $I$-antispherical double leaves generate the
  space  $\mathrm{Hom}^{\bullet}_{\mathcal{N}}(N_{\underline{x}},
  N_{\underline{y}})$ over $R_I$.

  The proof of the linear independence of light leaves and
  double leaves is very similar to the proof in the Hecke category, so
  we will be brief. First note that repeated application of the
  canonical decomposition (for $x \in {}^IW$)
  \[
K_x \cdot B_s = \begin{cases} K_x \oplus K_{xs} & \text{if $xs \in
    {}^IW$}, \\
  0 & \text{if $xs \notin {}^IW$} \end{cases}
    \]
    of Theorem \ref{Q}(\ref{4c}) gives a canonical decomposition
    \begin{equation}
      \label{eq:CanSplit}
      N_{\underline{x}}=\bigoplus_{\substack{\eb \subseteq \underline{x}\\
I\mhyphen\mathrm{antispherical}}} K_{\eb},
    \end{equation}
where $K_{\eb}:=K_{\underline{x}^{\eb}}$.
Now consider an anti-spherical light leaf
    \[
LL_{\underline{x}, \eb} : N_{\un{x}} \to N_{\un{w}}
\]
where $\un{w}$ is a reduced expression for $w \in {}^I W$. After
localizing and projecting to the canonical summand $K_w \subset
N_{\un{w}}$ we get maps
\[
p_{\fb}^{\eb} : K_{\fb} \to K_w
  \]
for each $I$-antispherical subexpression $\fb \subseteq \un{x}$ for
$w$. Because $Q_I = \Hom(K_{\fb}, K_{w})$
(see Theorem \ref{Q}(\ref{4d})) we may regard $p_{\fb}^{\eb}$ as
  an element of $Q_I$.

  \begin{prop} \label{prop:LLupper}
    We have that $p_{\fb}^{\eb}=0$ unless $\fb \le \eb$ in path
    dominance order. Moreover, $p_{\eb}^{\eb}$ is a non-zero
    product of the images of roots in $Q_I$, which is independent of the
    choice of light leaves.
  \end{prop}

  \begin{proof}
    The proof is very similar to the proof of \cite[Proposition
    6.6]{EW}. (Note that the $\alpha_k$ which appears in the proof of \cite[Proposition
    6.6]{EW} has non-zero image in $Q_I$ by the parabolic property.)
  \end{proof}

  In ${}_Q\NC$, the morphism $\LLL_{w,\fb,\eb} $ gives a coefficient
  $p^{\fb,\eb}_{\fb',\eb'}\in Q_I$ given by the inclusion of each
  standard summand $K_{\eb'}$ of $N_{\un{x}}$ and projection to each
  standard summand $K_{\fb'}$ of $N_{\un{y}}$, in the decomposition \eqref{eq:CanSplit}.
The following facts about these coefficients are easy consequences of
Proposition \ref{prop:LLupper} (see also the discussion in \cite[\S 6.3]{EW}):
\begin{itemize} \item $p^{\fb,\eb}_{\fb',\eb'}=0$ unless $({\un{x}},\eb')$ and $({\un{y}},\fb')$ express the same element $v$. This is a direct consequence of Theorem \ref{Q}(\ref{4d}).
\item $p^{\fb,\eb}_{\fb',\eb'}=0$ unless both $\eb'
\le \eb$ and $\fb' \le \fb$. This is a direct consequence of the construction of light leaves, the fact that the composition of the projection and the dot $$ B_{\id} \stackrel{\tikz[scale=-.3]{ \draw[red] (0,1) to (0,.4);
    \node[circle,fill,draw,inner sep=0mm,minimum size=1mm,color=red] at (0,.4) {};
}}{\longto}  B_s{\longto} K_s  $$ and the composition of the dot and the inclusion
$$K_s\ {\longto} B_s \stackrel{\tikz[scale=.3]{ \draw[red] (0,1) to (0,.4);
    \node[circle,fill,draw,inner sep=0mm,minimum size=1mm,color=red] at (0,.4) {};
}}{\longto}  B_{\id}  $$ are both zero, and again Theorem \ref{Q}(\ref{4d}).
\item The element $ p^{\fb,\eb}_{\fb,\eb}$ is invertible in $ Q_I$. Moreover, it is a product of roots, obeying a
simple formula independent of the choice of $\LL$ maps.
\end{itemize}

Consider the double path dominance order introduced in \S \ref{basic}, restricted to pairs of 01-sequences with the same fixed end-point. As we have seen, $\LLL$ maps satisfy upper-triangularity with respect to this partial order, with an invertible diagonal, thus giving linear independence of $\mathbb{LL}^{I\mathrm{as}}_{\underline{x},\underline{y}}$ over $R_I$. \end{proof}

\section{Categorification theorem}

Recall that in Section \ref{ASCN} we defined $D_x$ as the image of the indecomposable object
$B_x$ in $\mathcal{N}$. By the definition of $\mathcal{N}$ as a
quotient of additive categories, it is clear that the set
\[
\{ D_x \; |
\; x \in ^I\hspace{-0.13cm}W \}
\]
is a set of representatives for the
isomorphism classes of indecomposable objects of $\mathcal{N}$, up to
shift. Its image in $_Q\hspace{-0.03cm}\NC$ is of the form $$K_x\oplus
\bigoplus_{y<x}m_y\cdot K_y$$ with $m_y\in \mathbb{N}$ (note that the
$Q_x$ are  non-zero by Remark \ref{nonzero}). For any $x \in ^I\hspace{-0.13cm}W$ consider the full additive
subcategory
\[
\mathcal{N}_{\not \ge x} := \langle D_y(m) \; | \; y \not \ge x \text{
  and } m \in \ZM \rangle_{\oplus} \subseteq \mathcal{N},
\]
and the quotient (of additive categories)
\[
\mathcal{N}^{\ge x} := \mathcal{N} / \mathcal{N}_{\not \ge x}.
\]

\begin{lem} \label{lem:homI}
For any expression $\un{y}$,
$\mathrm{Hom}^{\bullet}_{\mathcal{N}^{\ge x}}( N_{\un{y}}, D_x)$ is a free graded
$R_I$-module, with basis  the (images of) the
$I$-antispherical light leaves corresponding to $I$-antispherical
subexpressions of $\un{y}$ expressing $x$. 
\end{lem}

\begin{proof} Let $\un{x}$ be a reduced expression for $x$. 
By Theorem \ref{DLT}, $\mathrm{Hom}^{\bullet}_{\mathcal{N}}( N_{\un{y}},
N_{\un{x}})$ is free over $R_I$ with basis given by $I$-antispherical
double leaves. However, when we pass to the quotient $\mathcal{N}^{\ge
  x},$ all double leaves with non-trivial upper light leaf factor
through an object in $\mathcal{N}_{\not \ge x}$ and are therefore
zero. We conclude that the claimed elements span
  $\mathrm{Hom}^{\bullet}_{\mathcal{N}^{\ge x}}( N_{\un{y}}, N_{\un{x}})$.

To see that they are linearly independent, consider the chain of
functors
\[
\NC \to \;  _Q \hspace{-0.03cm}\NC \to \; _Q\hspace{-0.03cm}\NC / \langle K_z
\;|\; z \not \ge x \rangle_{\oplus}
\]
where the first functor is given by localisation, and the second is
the quotient functor. If $y \not \ge x,$ the image of $D_y$  is zero,
and hence we obtain a functor
\[
\NC^{\ge x} \to \; _Q\hspace{-0.03cm}\NC / \langle K_z
\;|\; z \not \ge x \rangle_{\oplus}.
\]
By Proposition \ref{prop:LLupper} the maps are linearly independent on
the right hand side, and hence are on the left hand side too.
Thus the statement of the lemma is true for
$\mathrm{Hom}^{\bullet}_{\mathcal{N}^{\ge x}}( N_{\un{y}}, N_{\un{x}})$. Finally, $D_x$
and $N_{\un{x}}$ are isomorphic in $\mathcal{N}^{\ge x}$ and the lemma
follows.
\end{proof}

Because any object in $\mathcal{N}$ is a direct sum of shifts of
summands of $N_{\un{x}}$ we conclude that the space
$\mathrm{Hom}^{\bullet}_{\mathcal{N}^{\ge x}}(M, D_x)$ is a free graded $R_I$-module
for any object $M \in \mathcal{N}$. We define the \emph{diagrammatic
  character} as follows
\begin{align*}
  \ch : [ \mathcal{N}] &\to N \\
[M] & \mapsto \sum_{y \in {}^I W}  \grk
   \mathrm{Hom}^{\bullet}_{\mathcal{N}^{\ge y}}(M, D_y) n_y,
\end{align*}
where $\grk$ denotes graded rank.

\begin{thm}\label{central} Let $\Bbbk =
\RM$ and  $\hg$ be the dual geometric representation of $W$. The diagrammatic character gives an isomorphism
\[
\ch : [\mathcal{N}]\simto  N
\]
as $[\mathcal{H}] =  H$-modules. Under this isomorphism the
indecomposable object $D_x$ is mapped to the Kazhdan-Lusztig basis
$d_x$.
\end{thm}
\begin{proof}
It is clear that $\ch$ is a morphism of $ \ZM[v^{\pm
    1}]$-modules. As explained above, the set $\{ D_x | x \in  {}^I W \}$ gives
representatives for the indecomposable objects of $\NC$ up to shifts
and isomorphism. Hence
\[
[\NC] = \bigoplus \ZM[v^{\pm 1}] [D_x].
\]
On the other hand, it is immediate from the definition and Lemma \ref{lem:homI} that
\[
\ch([D_x]) = n_x + \sum_{y < x} n_{y,x}' n_y,
\]
for some $n_{y,x}' \in \ZM_{\ge 0}[v^{\pm 1}]$. We conclude that $\ch$ maps a
basis of $[\NC]$ to a basis of $N$, and hence is an isomorphism of
$\ZM[v^{\pm 1}]$-modules.

For any expression $\un{y} = s_1 \dots s_m,$ set $d_{\un{y}} := n_{\id}
\cdot b_{s_1} \dots b_{s_m}$.  Lemma \ref{lem:homI} combined with Equation
\eqref{rightactionN}  implies, by construction of  the $I$-antispherical light leaves,  that $\ch([N_{\un{y}}]) = d_{\un{y}}$. For
any $s \in S$  we have tautologically
\[
\ch( [N_{\un{y}}][ B_s]) = \ch([N_{\un{y}'}]) = d_{\un{y}'} = d_{\un{y}}
\cdot b_s,
\]
where $\un{y}' = s_1 \dots s_m s$. We conclude that $\ch$ is a map of
$[\mathcal{H}] =  H$-modules on the $\ZM[v^{\pm 1}]$-submodule generated by
$[N_{\un{y}}]$, where $\un{y}$ ranges over all expressions. However
this submodule is all of $[\NC]$ and hence, 
$\ch$ is an isomorphism of $[\mathcal{H}] =  H$-modules.

We will prove by induction  in $l(x)$ that $\ch([D_x])=d_x$, so let us
suppose that we know this equality for all $y$ such that $l(y)<l(x)$.
Let $\un{x}$ be a reduced expression for $x \in  {}^I W$. Then
in $\mathcal{H}$ we have
\[
B_{\un{x}} = B_x \oplus E,
\]
where $E$ is some self-dual object, all of whose indecomposable
summands are parametrized by $y < x$. By acting on $N_{\id}$ we
conclude that
\[
N_{\un{x}} = D_x \oplus \overline{E},
\]
where $\overline{E}$ is a self-dual combination of $D_y$ with $y < x$
and $y \in  {}^I W$. As observed above, we have $\ch([N_{\un{x}}]) =
d_{\un{x}}$ and so it is self-dual. By induction, $\ch([\overline{E}])$ is
self-dual. We deduce that $\ch([D_x])$ is self-dual as  well, as the
difference of two self-dual elements.

Finally, by the main theorem of \cite{EW2} (more precisely, see second
sentence following \cite[Theorem 3.6]{EW2})  we know that
$\mathrm{Hom}^{\bullet}_{\mathcal{N}^{\ge y}}(D_x, D_y )$ is generated in strictly
positive degrees for $y < x$. We conclude that the polynomials
$n'_{y,x}$ defined above actually satisfy $n'_{y,x} \in v\ZM[v]$ for
$y < x$. Hence by the uniqueness of the Kazhdan-Lusztig basis we
deduce that
\[
\ch([D_x]) = d_x.
\]
The theorem follows.
\end{proof}


\begin{cor}
  The anti-spherical Kazhdan-Lusztig polynomials $n_{y,x}$ have
  non-negative coefficients.
\end{cor}

Also, if $J \subset I \subset S$ it is immediate (either from Remark
\ref{bre} or from the fact that $\NC_I$ is a quotient of $\NC_J$) that we have a surjection
\[
\mathrm{Hom}^{\bullet}_{\mathcal{N}_J^{\ge y}}(D^J_x, D^J_y) \onto \mathrm{Hom}^{\bullet}_{\mathcal{N}_I^{\ge y}}(D^I_x, D^I_y),
\]
and we deduce:

\begin{cor}\label{Br}
Brenti's Monotonicity conjecture: $J\subseteq I$ implies that $n_{y,x}^I\leq n_{y,x}^J$, for $x,y\in {}^IW.$
\end{cor}

\bibliographystyle{myalpha}
\bibliography{gen}

\end{document}